\documentclass[a4paper, 12pt]{article}
\usepackage{amsmath}
\usepackage{amssymb}
\usepackage{amsthm}
\usepackage{amsfonts}
\usepackage[initials]{amsrefs}
\usepackage{mathtools}
\usepackage{mathrsfs}
\usepackage{empheq}
\usepackage{graphicx}

\usepackage{cases}
\usepackage{float}
\usepackage{color}
\usepackage{url}
\usepackage{titlefoot}
\usepackage{hyperref}
\usepackage{ulem}

\usepackage[
top=26truemm, 
bottom=25truemm, 
left=25.5truemm, 
right=25.5truemm
]{geometry}

\allowdisplaybreaks[1]
\newcommand{\dist}{\mathrm{dist}\,}          %distance function
\newcommand{\spt}{\mathrm{spt}\,}            %support

\newcommand{\capA}{\mathcal{A}}

\newcommand{\capC}{\mathcal{C}}

\newcommand{\capG}{\mathcal{G}}
\newcommand{\capH}{\mathcal{H}}

\newcommand{\capM}{\mathcal{M}}

\newcommand{\capS}{\mathcal{S}}

\newcommand{\capX}{\mathcal{X}}

\newcommand{\mathR}{\mathbb{R}}
\newcommand{\mathS}{\mathbb{S}}

\newcommand{\Pers}{\text{Per}_s}

\makeatletter

\@addtoreset{equation}{section}
\makeatother
\theoremstyle{mystyle}
\newtheorem{theorem}{Theorem}[section]
\newtheorem*{theorem*}{Theorem}
\newtheorem{lemma}[theorem]{Lemma}
\newtheorem{proposition}[theorem]{Proposition}

\theoremstyle{definition}

\theoremstyle{remark}
\newtheorem{remark}[theorem]{Remark}
\begin{document}
	\title{On the shape of hypersurfaces with boundary which have zero fractional mean curvature}
	
	\author{Fumihiko Onoue\footnote{Technische Universit\"at M\"unchen, Bolzmannstrasse 3, 85748 Garching, Germany.\\ 
			E-mail: {\tt fumihiko.onoue@tum.de}}
	}
	
	\date{\today}	
	
	\maketitle

\begin{abstract}
	We consider hypersurfaces with boundary in $\mathR^N$ that are the critical points of the fractional area introduced by Paroni, Podio-Guidugli, and Seguin in \cite{PPgS18}. In particular, we study the shape of such hypersurfaces in several simple settings. First we show that the critical points whose boundary is an $(N-2)$-sphere coincide with $(N-1)$-balls. Second we show that the critical points whose boundary is the union of two parallel $(N-2)$-spheres do not coincide with two parallel $(N-1)$-balls. Moreover, the interior of the critical points does not intersect the boundary of the convex hull of the two $(N-2)$-spheres, while it can happen in the situation considered by Dipierro, Onoue, and Valdinoci in \cite{DOV22}. We also obtain a quantitative bound which may tell us how different the critical points are from the two $(N-1)$-balls. Finally, in the same setting as in the second case, we show that, if the two parallel boundaries are far away from each other, then the critical points are disconnected and, if the two parallel boundaries are close to each other, then the boundaries are in the same connected component of the critical points when $N \geq 3$. Moreover, by computing the fractional mean curvature of a cone with the same boundaries as those of the critical points, we also obtain that the interior of the critical points does not touch the cone if the critical points are contained in either the inside or the outside of the cone.    
\end{abstract}

\tableofcontents

\section{Introduction}
Fractional minimal surfaces without boundary were first investigated by Caffarelli, Roquejoffre, and Savin in \cite{CRS10} and, since then, this topic has attracted many authors to study their geometric properties as an analogy of classical minimal surfaces. Roughly speaking, a fractional (or nonlocal) minimal surface without boundary is given as the boundary of a set which minimizes an energy functional defined by the pointwise interaction of a set and its complement. The typical interaction taken into account is scaling and translation invariant with some polynomial decay. Precisely, if $s\in(0,\,1)$ and $\Omega$ is an open set with smooth boundary, one of such standard energies of a set $E \subset \mathR^N$ relative to $\Omega$ is the so-called fractional perimeter in $\Omega$ and is defined by
\begin{equation}\label{fractionalPeriCRS}
	P_s(E ; \Omega) \coloneqq \int_{E \cap \Omega}\int_{E^c}\frac{dx\,dy}{|x-y|^{N+s}} + \int_{E \cap \Omega^c}\int_{E^c \cap \Omega}\frac{dx\,dy}{|x-y|^{N+s}}
\end{equation}
where we denote by $E^c$ the complement of $E$. With this notion, we say that a set $E \subset \mathR^N$ is a minimizer of $P_s$ relative to $\Omega$ if it holds that
\begin{equation*}
	P_s(E ; \Omega') \leq P_s(F; \Omega)
\end{equation*}
for any open bounded set $\Omega' \subset \Omega$ and any $F \subset \mathR^N$ with $F \setminus \Omega' = E \setminus \Omega'$. The existence and regularity of such minimizers was shown by Caffarelli, Roquejoffre, and Savin in \cite{CRS10}. Moreover, they showed in \cite{CRS10} that, if a set $E \subset \mathR^N$ is a minimizer of $P_s$, then the following Euler-Lagrange equation holds in the viscosity sense:
\begin{equation}\label{eulerLagrangeEqCRS}
	\int_{\mathR^N} \frac{\chi_{E^c}(y) - \chi_{E}(y)}{|y-x|^{N+s}} \,dy = 0
\end{equation}
for $x \in \partial E$. The integral in \eqref{eulerLagrangeEqCRS} is intended in the Cauchy principal value sense. This can be regarded as a nonlocal counterpart of the classical minimal surface equation and the left-hand side in \eqref{eulerLagrangeEqCRS} is the so-called fractional mean curvature on the boundary $\partial E$. Dipierro, Savin, and Valdinoci in particular have revealed many properties which classical minimal surfaces cannot possess (see, for instance, \cite{DSV17, DSV20} for the detail). In addition, many authors have studied the fractional(nonlocal) minimal surfaces or minimal graphs for more than a decade since the fractional(nonlocal) minimal surfaces appear in many other topics in which a long-range interaction is involved (see \cite{CDV17, SavVal12}). For further discussions about the geometric features of fractional(nonlocal) minimal surfaces without boundary, we refer to \cite{CafVal11, CafVal13, DdPW18, BLV19, BDLV20, CoLu21, DOV22, BDV23, DSV23MaxPrin, DSV23Bdry, DSV16}. 

Quite recently, motivated by some mathematical modelling of thin elastic structures, Paroni, Podio-Guidugli, and Seguin in \cite{PPgS18} introduced a new notion of fractional areas and fractional mean curvatures for smooth manifolds which are not necessarily closed in the following way: let $\Omega \subset \mathR^N$ be a bounded domain and let $\capM \subset \Omega$ be any $(N-1)$-dimensional compact smooth manifold with or without boundary. Then the fractional area of $\capM$ relative to $\Omega$ is defined by
\begin{equation}\label{fractionalPeriPPgS}
	\Pers(\capM ; \Omega) \coloneqq c_N\iint_{\capX(\capM)} \frac{ \max\{\chi_{\Omega}(x),\,\chi_{\Omega}(y)\} }{|x-y|^{N+s}} \,dx\,dy
\end{equation}
where $c_N$ is some positive dimensional constant and $\capX(\capM)$ is a set of all pairs $(x,\,y) \in \mathR^N \times \mathR^N$ such that the segment $[x; \, y]$ with two end points $x$ and $y$ has an odd number of cross intersections with $\capM$ and $[x; \, y]$ is not tangent to $\capM$. Note that the presence of the term $\max\{\chi_{\Omega}(x),\,\chi_{\Omega}(y)\}$ in \eqref{fractionalPeriPPgS} is necessary to ensure that the integral converges whenever $\partial \capM \neq \emptyset$.
%The authors in \cite{PPgS18} also introduced a localized version of \eqref{fractionalPeriPPgS}, i.e., the fractional area $\Pers(\capM; \Omega)$ of a (bounded or unbounded) manifold $\capM$ relative to an open bounded set $\Omega \subset \mathR^N$ is defined in the following way: 

As is explained in \cite{PPgS18}, if a $(N-1)$-dimensional smooth manifold $\capM$ satisfies $\capM = \partial E$ for some set $E \subset \mathR^N$, then two notions \eqref{fractionalPeriCRS} and \eqref{fractionalPeriPPgS} are equivalent, i.e., it holds that
\begin{equation*}
	\Pers(\capM; \Omega) = P_s(E; \Omega).
\end{equation*}
Interestingly, Paroni, Podio-Guidugli, and Seguin also proved in \cite[Theorem 3.3]{PPgS18} that $(1-s)\Pers(\capM; \Omega) \to \capH^{N-1}(\capM)$ as $s \uparrow 1$ for a compact $(N-1)$-dimensional $C^1$ manifold $\capM$ contained in a bounded domain $\Omega$, as it happens for $P_s$ in \eqref{fractionalPeriCRS} (see \cite{ADPM11, CaVa11}). See \cite{Seguin20, PPgS22, MihSeg23} for further discussions on $\Pers$.

This manuscript is devoted to develop the theory of the fractional area $\Pers$ for manifolds with boundary. In particular, we aim to investigate the shape and topology of critical points of $\Pers$. Here the critical point of $\Pers$ is defined by a smooth manifold such that the first variation of $\Pers$ vanishes with respect to a perturbation associated with the unit normal vector of that manifold (in the sequel, we will call this perturbation ``normal variations''). The authors in \cite{PPgS18} obtained a necessary and sufficient condition for the vanishing of the first variation for manifolds as follows: let $\capM$ be an orientable compact smooth manifold with or without boundary and assume that $\capM$ is contained in a bounded domain $\Omega \subset \mathR^N$. Then it holds that
\begin{equation}\label{eulerLagrangeEqPPgS}
	\delta \Pers(\capM; \Omega) = 0  \quad \iff \quad H_{\capM,s}(z) = 0 \quad \text{for any $z \in \capM$}.
\end{equation}
Here we denote by $\delta \Pers(\capM; \Omega)$ the first variation of $\capM$ under normal variations and $H_{\capM,s}$ is the fractional mean curvature associated with $\Pers$ which is defined by
\begin{equation*}
	H_{\capM,s}(z) \coloneqq c_N \, \int_{\mathR^N} \frac{\chi_{\capA_i(z)}(y) - \chi_{\capA_e(z)}(y) }{|y-z|^{N+s}} \,dy
\end{equation*} 
for any $z \in \capM$ where $c_N$ is as in \eqref{fractionalPeriPPgS} and the sets $\capA_i(z)$ and $\capA_e(z)$ are defined by
\begin{align}
	& \capA_i(z) \coloneqq \{ y \in \mathR^N \mid \text{either $(z,y) \in \capX(\capM) \, \& \, (z-y) \cdot \nu_{\capM}(z) < 0$}  \nonumber\\
	& \qquad \quad \qquad \quad \qquad \quad \qquad \text{or $(z,y) \not\in \capX(\capM) \, \& \, (z-y) \cdot \nu_{\capM}(z) > 0$} \} \label{interior} \\
	& \capA_e(z) \coloneqq \{ y \in \mathR^N \mid \text{either $(z,y) \in \capX(\capM) \, \& \, (z-y) \cdot \nu_{\capM}(z) > 0$}  \nonumber\\
	& \qquad \quad \qquad \quad \qquad \quad \qquad \text{or $(z,y) \not\in \capX(\capM) \, \& \, (z-y) \cdot \nu_{\capM}(z) < 0$} \}. \label{exterior}
\end{align}
The sets $\capA_i(z)$ and $\capA_e(z)$ can be regarded as the ``interior'' and ``exterior'' of $\capM$ relative to the point $z$, respectively, and these sets are determined uniquely once the unit normal vector of $\capM$ at $z$ is specified. See \cite{PPgS18} for more discussions on the notions. Note that, if a manifold is not orientable, then the unit normal vector of the manifold cannot be determined uniquely and neither can the ``interior'' $\capA_i$ and ``exterior'' $\capA_e$. Moreover, in this paper, we require the $C^{1,\alpha}$-regularity with $\alpha > s$ of hypersurfaces so that the fractional mean curvatures are finite everywhere.

The study of critical points or fractional minimal surfaces with boundary can be related to the classical problem on free boundary minimal surfaces in differential geometry. One of the main topics in the problem is to determine the shape of a manifold $\Sigma$ (embedded or immersed) in another smooth manifold $\capS$ such that $\Sigma$ minimizes its area in $\capS$ and $\partial \Sigma \subset \partial \capS$ with some topological constraints. The study of this classical problem was first considered by R. Courant in \cite{Courant40} in 1940 and, since then, a lot of authors have been intensively working on this topic. See, for instance, \cite{Lewy51, MeeYau80, Smyth84, Nitsche85, Ros08} for the detail. We also refer the readers to two surveys: \cite{Hildebrandt86} for classical works and \cite{Li20} for more recent results. The references here are obviously not exhausted. 

As an analogy of the classical free boundary minimal surfaces, it is natural to consider a fractional(nonlocal) version of free boundary minimal surfaces; however, the nonlocal version is not understood so far because, to our knowledge, suitable notions of fractional areas for manifolds with boundary had not been considered until Paroni, Podio-Guidugli, and Seguin in \cite{PPgS18} introduced the notion of $\Pers$ in \eqref{fractionalPeriPPgS}. To tackle the nonlocal version of the free boundary minimal surface problem, it is important to understand the geometric properties of critical points of $\Pers$.

%We finally remark that the existence of minimizers with boundary of $\Pers$ is not proved and is still open.     

%For our purposes, we will often denote coordinates in $\mathR^N$ by $x= (x',\,x_N) \in \mathR^{N-1} \times \mathR$.
%and will focus on the case of cylindrical domains of the form
%\begin{equation}\label{localDomainwithLengtha}
%	\Omega_a = \{(x',\,x_N) \in \mathR^{N-1} \times \mathR \mid |x'| < a\}
%\end{equation} 
%for any $a > 0$.
%In the framework of classical minimal surfaces, it is easy to see that the only minimizer (or the only critical point) of the classical perimeter in our setting is a hyperplane whose boundary is $\Gamma$. Moreover, in the framework of fractional minimal surfaces without boundary, in some proper situation, one can also prove that the minimizer of the fractional perimeter defined in \eqref{fractionalPeriCRS} is also a hyperplane. Regarding the critical points or minimizers with boundary of the fractional area $\Pers$, it is not so obvious whether they must be a hyperplane whose boundary is $\Gamma$. Since the notion introduced in \cite{PPgS18} is relatively new, any rigorous answer even to this simple question is not known so far. In this paper, we are able to give a positive answer to this question. As we may expect,

Given the importance of critical points of $\Pers$ from the above perspective, it is desirable to develop some intuition about their geometric features. To do this, since it is quite difficult to have explicit solutions which entirely describe critical points or minimizers of $\Pers$, it is often convenient to study simplified cases in which the boundary of the critical points has some special characteristics. In this paper, we basically consider three cases: the first is that the boundary of critical points in $\mathR^N$ lies in a hyperplane and is homeomorphic to $\mathS^{N-2}$ (our result is also true if the boundary is not always homeomorphic to $\mathS^{N-2}$). The second is that the boundary is the union of two distinct parallel and co-axial manifolds each of which lies in a hyperplane, is homeomorphic to $\mathS^{N-2}$, and has distance of $d$ from another boundary. The last is that the distance $d$ is sufficiently large or sufficiently small.
  
Our first goal in this paper is to determine the shape of critical points of $\Pers$ whose boundary lies on a hyperplane. Precisely, we first define a set $\capC \subset \mathR^N$ as
\begin{equation}\label{boundaryCylinder}
	\capC \coloneqq \capG \times \mathR
\end{equation}
where $\capG$ is a non-empty bounded open subset of $\{x_N=0\}$ with a smooth boundary. Then we define an $(N-2)$-dimensional smooth manifold $\Gamma$ as
\begin{equation}\label{boundaryData}
	\Gamma_0 \coloneqq  \partial \capC \cap \{x_N=0\} \, (= \partial \capG \times \{0\}).
\end{equation}
Assume that $\capM \subset \mathR^N$ is an orientable compact $(N-1)$-dimensional $C^{1,\alpha}$ manifold with $\partial \capM = \Gamma_0$ and that $\capM$ is a critical point of $\Pers$. Note that the orientability of $\capM$ implies the orientability of $\partial \capM = \Gamma_0$. Then, as our first theorem, we aim to rigorously prove that $\capM$ must coincide with $\capC \cap \{x_N=0\}$, as we can intuitively expect this to be true. 

\begin{theorem}\label{theoremMinimalityHyperplane}
	Let $s \in (0,\,1)$. Let $\Gamma_0$ be as in \eqref{boundaryData}. Let $\capM$ be an orientable compact $(N-1)$-dimensional $C^{1,\alpha}$ manifold with $\partial \capM = \Gamma_0$. If $\capM$ is a critical point of $\Pers$ under normal variations, then $\capM$ is a hyperplane lying on $\{ x_N=0 \}$, i.e., 
	\begin{equation*}
		\capM = \overline{\capC} \cap \{x_N=0\} \, (= \overline{\capG} \times \{0\}). 
	\end{equation*}
\end{theorem}
%Note that we can prove the same result as in Theorem \ref{theoremMinimalityHyperplane} even if we replace $\Gamma_a$ with an $(N-2)$-dimensional smooth closed manifold $\Gamma$ lying on $\{(x' ,\,x_N) \mid x_N = 0\}$ because our idea of the proof may also work for $\Gamma$. In this setting, the critical point $\capM$ of $\Pers$ coincides with the set on $\{x_N = 0\}$ enclosed by $\Gamma$.  
%\fumihiko{In fact, I would not use the minimality of $\capM$ since the essential point is that $\capM$ satisfies some equation pointwisely and Theorem 4.1 in \cite{PPgS18} says about surfaces at which the first variation vanishes.}

Our second goal in this paper is to study the shape of critical points of $\Pers$ whose boundary consists of two disjoint components. The problem setting in the second theorem is as follows: we define two distinct compact $(N-2)$-dimensional smooth manifolds $\Gamma_1$ and $\Gamma_2$ by 
\begin{align}
	\Gamma_1 \coloneqq \partial \capC \cap \{x_N = h_1\} \quad \text{and} \quad \Gamma_2 \coloneqq \partial \capC \cap \{x_N = h_2\}, \label{boundary2}
\end{align}
where $\capC$ is as in \eqref{boundaryCylinder} and $h_1$ and $h_2$ are given constants with $h_2 < h_1 $. Then a critical point exhibit a different shape from a hyperplane. Precisely we prove
\begin{theorem}\label{theoremTwoHyperplanesNotMinimal}
	Let $s \in (0,\,1)$. Let $\Gamma_1$ and $\Gamma_2$ be as in \eqref{boundary2} and let $\capM$ be an orientable compact $(N-1)$-dimensional $C^{1,\alpha}$ manifold with $\partial \capM = \Gamma_1 \cup \Gamma_2$. Assume that $\capC$ is convex where $\capC$ is as in \eqref{boundaryCylinder}. If $\capM$ is a critical point of $\Pers$ under normal variations, then any connected component of $\capM$ is neither $C_1$ nor $C_2$ where we define
	\begin{equation*}
		C_1 \coloneqq \overline{\capC} \cap \{x_N=h_1\} \quad \text{and} \quad C_2 \coloneqq \overline{\capC} \cap \{x_N=h_2\}.
	\end{equation*}
	In particular, $\capM \neq C_1 \cup C_2$ 
	Moreover, $\capM \setminus \partial \capM$ does not intersect $\partial \capC = \partial \capG \times \mathR$.  
\end{theorem}
We remark that, by using a cone whose boundary is $\Gamma_1 \cup \Gamma_2$ as in Theorem \ref{theoremTwoHyperplanesNotMinimal} with $h_1=1$ and $h_2=-1$, we can further detect how the critical points behave. See Subsection \ref{subsectionFurtherStudy} of Section \ref{sectionProofTheorem2} for the detail. 
%Interestingly, Theorem \ref{theoremTwoHyperplanesNotMinimal} indicates that the ``stickiness'' of the minimizers of $\Pers$ to the cylinder $\capC$
%
%, in dimension 3, critical points or minimizers of $\Pers$ constrained to two nearby parallel and co-axial circumstances (the case that $h_1-h_2>0$ is sufficiently small) develop necks of catenoids, as is the case in the classical minimal surfaces. On the other hand, as is shown in \cite[Theorem 1.1]{DOV22}, there are no possibilities that minimizers of $P_s$ under the same circumstances develop necks of catenoids in any dimension.
%We may actually expect that the first result in Theorem \ref{theoremTwoHyperplanesNotMinimal} holds true since the author in this paper with Dipierro and Valdinoci in \cite{DOV22} has already observed a similar phenomenon for minimizers of $P_s$ relative to some open cylinder; however, what is different from the result in \cite{DOV22} is that our critical point (or minimizers) cannot have the ``stickiness'' property which was introduced by Dipierro, Savin, and Valdinoci in \cite{DSV16, DSV17, DSV20}. See also \cite[Proposition 4.1]{DOV22} for the detail.

%The idea of the proof is basically the same as the one of the proof of Theorem \ref{theoremMinimalityHyperplane}. When we prove the latter result in Theorem \ref{theoremTwoHyperplanesNotMinimal}, the convexity assumption on $\capC$ will be necessary for us to use the ``sliding method'' as in the proof of Theorem \ref{theoremMinimalityHyperplane}.

Our third goal is to further study the shape and, in particular, the topology of critical points of $\Pers$ in the same situation as the one in Theorem \ref{theoremTwoHyperplanesNotMinimal}. Precisely, taking $\Gamma_1$ and $\Gamma_2$ as in Theorem 1.2 with $d \coloneqq h_1-h_2>0$, we will see what critical points of $\Pers$ under normal variations look like in terms of connectedness if $d$ is sufficiently large or sufficiently small.

To reach the third goal, we first show the following lemma which somehow tells us how different critical points are from hyperplanes.
\begin{lemma}\label{theoremPoppingCriticalPtwithTwoBdry}
	Let $s \in (0,\,1)$ and $d>0$. Let $\Gamma_1$ and $\Gamma_2$ be as in \eqref{boundary2} with $h_1 = 0$ and $h_2 = -d$. Assume that $\capC$ is convex where $\capC$ is as in \eqref{boundaryCylinder}. Then there exists a constant $\varepsilon_0 > 0$, depending only on $N$, $s$, and $d$, such that the following holds: let $\capM$ be an orientable compact $(N-1)$-dimensional $C^{1,\alpha}$ manifold with $\partial \capM = \Gamma_1 \cup \Gamma_2$. If $\capM$ is a critical point of $\Pers$ under normal variations, then a set enclosed by $\capM$ and the union of $\capC \cap \{x_N=0\}$ and $\capC \cap \{x_N=-d\}$ contains two half-balls 
	\begin{equation*}
		B^{-}_{\varepsilon_0}(0) \coloneqq \{x \in B_{\varepsilon_0}(0) \mid x_N < 0 \}  \quad \text{and} \quad B^{+}_{\varepsilon_0}(p_{d}) \coloneqq \{x \in B_{\varepsilon_0}(p_{d}) \mid x_N > -d \}
	\end{equation*}
	where $p_{d} \coloneqq (0,\,-d) \in \mathR^{N-1} \times \mathR$.
\end{lemma}
% and $\psi: (0,\,+\infty) \to \mathR$ is an increasing function on $\mathR$, depending only on $N$, with $\lim_{t \downarrow 0}\psi(t)=0$
To favor the intuition, a sketch of our critical points is given in Figure \ref{figureTheoremBump}.

\begin{figure}[h]
\begin{center}
		\includegraphics[keepaspectratio, scale=0.55]{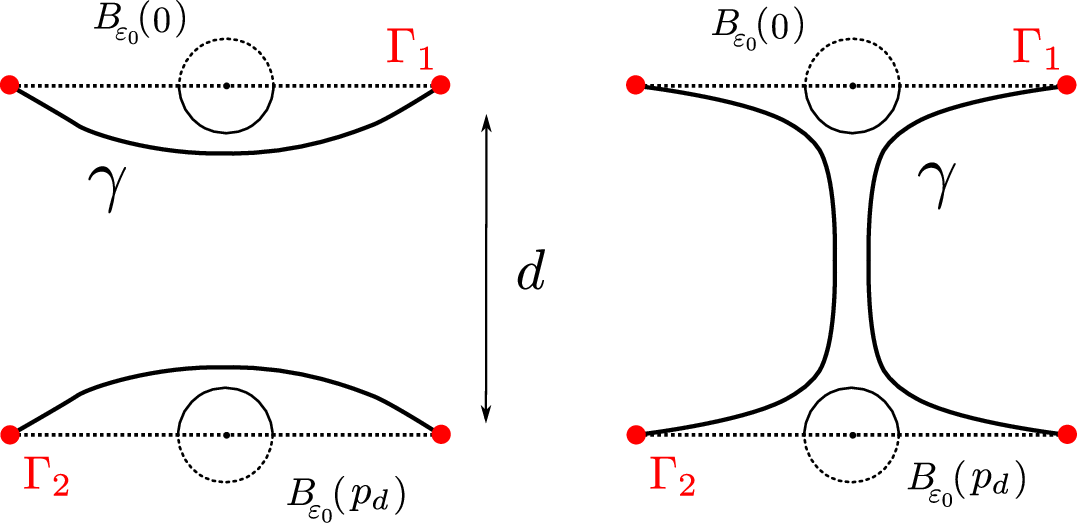}
\end{center}
	\caption{Two possible situations in dimension 2 in Theorem \ref{theoremPoppingCriticalPtwithTwoBdry} in which the `interior'' or ``exterior'' of the critical point $\gamma$ with $\partial \gamma = \Gamma_1 \cup \Gamma_2$ contains two half-balls.}
	\label{figureTheoremBump}
\end{figure}

As a result of Lemma \ref{theoremPoppingCriticalPtwithTwoBdry}, we prove that, if the distance $d$ between two parallel and co-axial boundaries is sufficiently small, then any critical point is connected in the sense that the two boundaries are in the same connected component when $N \geq 3$. Moreover, when $N=2$, any critical point is disconnected and its two distinct connected components should look like the right-hand side of Figure \ref{figureTheoremBump} with $0< d \ll 1$. 

Precisely, our third theorem is as follows.
\begin{theorem}\label{theoremConnectednessCriticalPtsN3}
	Let $s \in (0,\,1)$. Let $\Gamma_1$ and $\Gamma_2$ be as in Lemma \ref{theoremPoppingCriticalPtwithTwoBdry}. Assume that $\capC$ is convex where $\capC$ is as in \eqref{boundaryCylinder}. Then there exists $d_0=d_0(N,s)>0$ such that the following holds: for any $d \in (0,\, d_0)$, we take any orientable compact $(N-1)$-dimensional $C^{1,\alpha}$ manifold $\capM \subset \mathR^N$ with $\partial \capM = \Gamma_1 \cup \Gamma_2$. If $\capM$ is a critical point of $\Pers$ under normal variations, then $\Gamma_1$ and $\Gamma_2$ are in the same connected component of $\capM$ if $N \geq 3$ and $\capM$ is disconnected if $N=2$.
	
	Moreover, when $N=2$, there exist two distinct connected components $\capM_1$ and $\capM_2$ of $\capM$ such that $\dist(\capM_1, \capM_2) \geq c$ with some constant $c>0$, depending only on $N$ and $s$, and $\partial \capM_i$ intersects both $\Gamma_1$ and $\Gamma_2$ for each $i \in \{1,2\}$.  
\end{theorem} 

%We remark that, in dimension 2, the connectedness in Theorem \ref{theoremConnectednessCriticalPtsN3} fails because any orientable curve whose boundary consists of four different points should be the union of two disjoint curves.

%$\neq \emptyset$ and $\partial \capM_i \cap \Gamma_2 \neq \emptyset$ 

As a counterpart of Theorem \ref{theoremConnectednessCriticalPtsN3}, we prove that, if the distance $d$ between two parallel and co-axial boundaries is sufficiently large, then any critical point is disconnected in any dimensions and it should look like the left-hand side of Figure \ref{figureTheoremBump} with $d \gg 1$.

Our last theorem is as follows.
\begin{theorem}\label{theoremDisconnectednessCriticalPts}
	 Let $s \in (0,\,1)$. Let $\Gamma_1$ and $\Gamma_2$ be as in Lemma \ref{theoremPoppingCriticalPtwithTwoBdry}. Assume that $\capC$ is convex where $\capC$ is as in \eqref{boundaryCylinder}. Then there exists $d_1=d_1(N,s)>0$ such that the following holds: we assume that, for any $d >d_1$, $\capM \subset \mathR^N$ is any orientable compact $(N-1)$-dimensional $C^{1,\alpha}$ manifold with $\partial \capM = \Gamma_1 \cup \Gamma_2$. If $\capM$ is a critical point of $\Pers$ under normal variations, then $\capM$ is disconnected. 
	 
	 Moreover, there exist two disjoint connected components $\capM_1$ and $\capM_2$ of $\capM$ such that $\partial \capM_i = \Gamma_i$ for any $i \in \{1,\,2\}$.
\end{theorem}

The topological properties in Theorem \ref{theoremConnectednessCriticalPtsN3} and \ref{theoremDisconnectednessCriticalPts} could be expected to be true because Dipierro, Valdinoci, and the author of this paper obtained similar results in \cite{DOV22} on the topology of fractional minimal surfaces without boundary in the similar situations. On one hand, they showed that minimizers of $P_s$ in a given cylinder coincides with the cylinder itself for sufficiently small $d$ where $d$ is the distance between two disjoint parallel and co-axial external(boundary) data. On the other hand, they showed that minimizers of $P_s$ in the cylinder are disconnected for sufficiently large $d$.

Interestingly, however, we show in Theorem \ref{theoremTwoHyperplanesNotMinimal} that the critical points (not necessarily fractional area-minimizing) cannot touch the boundary of the cylinder $\capC$ no mater what distance two parallel and co-axial boundaries have, while it is shown in \cite{DOV22} that minimizers of $P_s$ in a cylinder favorably stick to the boundary of the cylinder if $N=2$ and $d$ is large or if $N \geq 2$ and $d$ is small. Moreover, our results together with Remark \ref{remarkCriticalPtsDsmall} of Section \ref{sectionTheorem3} possibly indicate that critical points of $\Pers$ with two nearby parallel and co-axial compact boundaries might develop necks of catenoids, while this is not the case with fractional minimal surfaces considered in \cite{DOV22}. We remark that the existence of fractional minimal catenoids without boundary in $\mathR^3$ was shown by D\'avila, Del Pino, and Wei in \cite{DdPW18} if $s$ is close to 1.

%The idea of the proof of our main results is basically to construct a suitable barrier which may prevent the fractional mean curvature of critical points from vanishing whenever it touched the critical points. 

%The idea of the proof is to construct a ``barrier'', whose fractional mean curvature is strictly positive or negative, against a critical point. The construction is inspired by the one shown by Dipierro, Savin, and Valdinoci in \cite{DSV23Bdry} (see also \cite[Proof of Proposition 4.1]{DOV22}).

The organization of this paper is as follows: in Section \ref{sectionProofTheorem1}, we prove Theorem \ref{theoremMinimalityHyperplane} by ``sliding'' a hyperplane until it touches critical points (see the proof of Theorem \ref{theoremMinimalityHyperplane} for the detail). In Section \ref{sectionProofTheorem2}, we first give the proof of Theorem \ref{theoremTwoHyperplanesNotMinimal} and then we study further properties of critical points of $\Pers$, computing the fractional mean curvature of a cone passing through the boundary of critical points. In Section \ref{sectionTheorem3}, we first give the proof of Lemma \ref{theoremPoppingCriticalPtwithTwoBdry} by constructing a suitable barrier and then, by using this lemma, we prove Theorem \ref{theoremConnectednessCriticalPtsN3}. Moreover, in Section \ref{sectionTheorem3}, we also prove Theorem \ref{theoremDisconnectednessCriticalPts} by means of the ``sliding method'' (see Section \ref{sectionTheorem3} for the detail).

%In Appendix \ref{sectionTopManiWithDisjointBoundaries}, we give a proof for some basic fact on the topology of a manifold with two disjoint boundaries which lies in the cylinder with a hole.
%In Appendix \ref{sectionSymmMini}, we show that, whenever their boundary is invariant with respect to an isometry, minimizers of $\Pers$ with boundary also possess the same invariance. This result is not for our main purpose; however we believe it is of independent interest. 

\section{Proof of Theorem \ref{theoremMinimalityHyperplane}}\label{sectionProofTheorem1}
In this section, we prove Theorem \ref{theoremMinimalityHyperplane}. The idea of the proof is inspired by the so-called sliding method introduced by Dipierro, Savin, and Valdinoci in \cite{DSV17}. They developed this method in order to investigate the shape of fractional(nonlocal) minimal surfaces (see also \cites{DSV16, DSV20, DOV22} for further discussions). 

We proceeds with the proof in the following way: we slide a hyperplane, parallel to $\capC \cap \{x_N=0\}$, from below or above until it touches $\capM$ and assume by contradiction that there exists a touching point in $(\capC \cap \{x_N=0\})^c$. At the touching point $q$, we obtain the Euler-Lagrange equation \eqref{eulerLagrangeEqPPgS}. Then, taking into account all the contributions from the ``interior'' $\capA_i(q)$ and the ``exterior'' $\capA_e(q)$ of $\capM$, we can observe that the contribution from either $\capA_i(q)$ or $\capA_e(q)$ turns out to be strictly larger than that from the other region. This contradicts the Euler-Lagrange equation.
 
%This is because most region of the set $\capA_e(q)$ (or $\capA_i(q)$) would be included in the cone whose vertex is $q$ and passes through $\partial \capM = \Gamma$ due to the assumption.

\begin{proof}[Proof of Theorem \ref{theoremMinimalityHyperplane}]

We first define a hyperplane $H_{\lambda} \coloneqq \{(x',\,x_N) \mid x_N = \lambda\}$ and two half-spaces 
\begin{equation}
	H_{\lambda}^{+} \coloneqq \{(x',\,x_N) \mid x_N > \lambda\} \quad \text{and} \quad  H_{\lambda}^{-} \coloneqq \{(x',\,x_N) \mid x_N < \lambda\}
\end{equation}
for $\lambda \in \mathR$. We set $P_{\lambda} : \mathR^N \to \mathR^N$ as the reflection map with respect to $H_{\lambda}$ for $\lambda\in\mathR$ and set $x_{\lambda} \coloneqq P_{\lambda}(x)$ for any $x \in \mathR^N$. Moreover, we denote by $C_{\Gamma_0}(q)$ a (filled) cone with vertex $q$ whose boundary passes through $\Gamma_0$, that is, $\{|x'|=a, \, x_N = 0\} \cap \partial C_{\Gamma_0}(q) = \Gamma_0$. We further set $C^{\lambda}_{\Gamma_0}(q) \coloneqq P_{\lambda}(C_{\Gamma_0}(q))$.

Now let $\capM \subset \mathR^N$ be the critical point chosen in Theorem \ref{theoremMinimalityHyperplane}. The minimizer $\capM$ is bounded. Hence, we can slide the hyperplane $H_{\lambda}$ from below until it touches the minimizer $\capM$. Our result in Theorem \ref{theoremMinimalityHyperplane} states that this touching does not occur in $H_{0}^{-} \cup H_{0}^{+}$ and thus, we assume by contradiction that there exist a constant $\lambda_0 < 0$ and a point $q \in \capM \cap \Omega$ such that 
\begin{equation*}
	T_{q} \capM = H_{\lambda_0} \quad  \text{and} \quad H_{\lambda_0}^- \cap \capM = \emptyset
\end{equation*}  
where $T_{q}\capM$ is a tangent space of $\capM$ at $q$. Due to the symmetry of our setting, we can conduct the same argument that we will show below in the case that we slide the hyperplane from above and the touching occurs in $H_{0}^{+}$. Hence, it is sufficient to show the proof in the case that the touching occurs in $H_{0}^{-}$. See also Figure \ref{figure1} for the situation that we consider in dimension 2.
%Since $q$ is in $\Omega$, we can choose a constant $\delta > 0$ such that $B_{\delta}(q) \subset \Omega$.

\begin{figure}[h]
\begin{center}
	\includegraphics[keepaspectratio, scale=0.60]{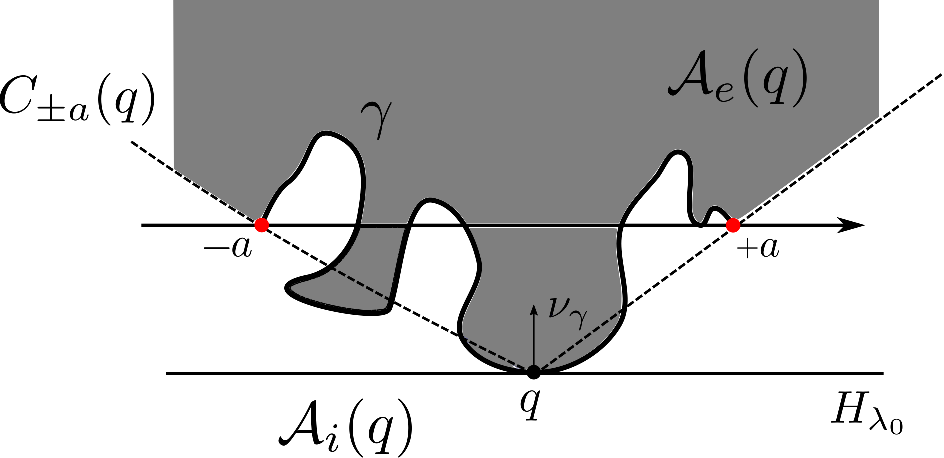}
\end{center}
\caption{The situation in dimension 2 in which the critical point $\capM = \gamma$ is a $C^{1,\alpha}$ curve with $\partial \gamma = \{\pm a\}$ where $\pm a \coloneqq (\pm a, \,0)$. The set $\capA_e(q)$ is shown in dark gray, the set $\capA_i(q)$ in white. The dashed lines represent the boundary of the cone $C_{\pm a}(q)$.}
\label{figure1}
\end{figure}

Since $\capM$ is an orientable compact critical point of $\Pers$, which means the vanishing of the first variation of $\Pers$ at $\capM$, and since $q \in \capM$, we obtain, from \eqref{eulerLagrangeEqPPgS}, that
\begin{equation}\label{eulerLagrangeMinimality}
		0 = H_{\capM,s}(q) \coloneqq c_N \, \int_{\mathR^N} \frac{\chi_{\capA_i(q)}(y) - \chi_{\capA_e(q)}(y) }{|y-q|^{N+s}} \,dy
\end{equation}
where the sets $\capA_e(q)$ and $\capA_i(q)$ are defined as in \eqref{interior} and \eqref{exterior}. We consider all the contributions from $\capA_e(q)$ and $\capA_i(q)$ in detail and show that the singular integral in the right-hand side of \eqref{eulerLagrangeMinimality} is strictly positive, which is a contradiction.

Indeed, since $C_{\Gamma_0}(q) \subset H_{\lambda_0}^{+}$ and $H_{\lambda_0}$ is tangential to $\capM$, we have that $P_{\lambda_0}(\capA_e(q)) \subset H_{\lambda_0}^{-} \subset \capA_i(q)$. This implies that $\mathR^N = \capA_e(q) \cup P_{\lambda_0}(\capA_e(q)) \cup \capA_i(q) \setminus  P_{\lambda_0}(\capA_e(q))$, up to negligible sets, and thus we can compute the fractional mean curvature $H_{\capM, s}$ at $q$ as follows:
\begin{align}
	c_N^{-1} H_{\capM, s}(q) &= \int_{\capA_e(q)} \frac{\chi_{\capA_i(q)}(y) - \chi_{\capA_e(q)}(y) }{|y-q|^{N+s}} \,dy + \int_{ P_{\lambda_0}(\capA_e(q)) } \frac{\chi_{\capA_i(q)}(y) - \chi_{\capA_e(q)}(y) }{|y-q|^{N+s}} \,dy \nonumber\\
	&\qquad + \int_{\capA_i(q) \setminus  P_{\lambda_0}(\capA_e(q))} \frac{\chi_{\capA_i(q)}(y) - \chi_{\capA_e(q)}(y) }{|y-q|^{N+s}} \,dy \nonumber\\
	&= \int_{\capA_e(q)} \frac{-1}{|y-q|^{N+s}} \,dy + \int_{ P_{\lambda_0}(\capA_e(q)) }\frac{1}{|y-q|^{N+s}} \,dy \nonumber\\
	&\qquad  + \int_{\capA_i(q) \setminus  P_{\lambda_0}(\capA_e(q))} \frac{1}{|y-q|^{N+s}} \,dy. \label{computationFracMC}
\end{align}
From the change of variables $y \mapsto P_{\lambda_0}(y)$ and the definition of $P_{\lambda_0}$, we have 
\begin{equation}\label{symmetryContribution}
	\int_{ P_{\lambda_0}(\capA_e(q)) }\frac{1}{|y-q|^{N+s}} \,dy = \int_{\capA_e(q)} \frac{1}{|y-q|^{N+s}} \,dy.
\end{equation}

\begin{figure}[h]
	\begin{center}
		\includegraphics[keepaspectratio, scale=0.55]{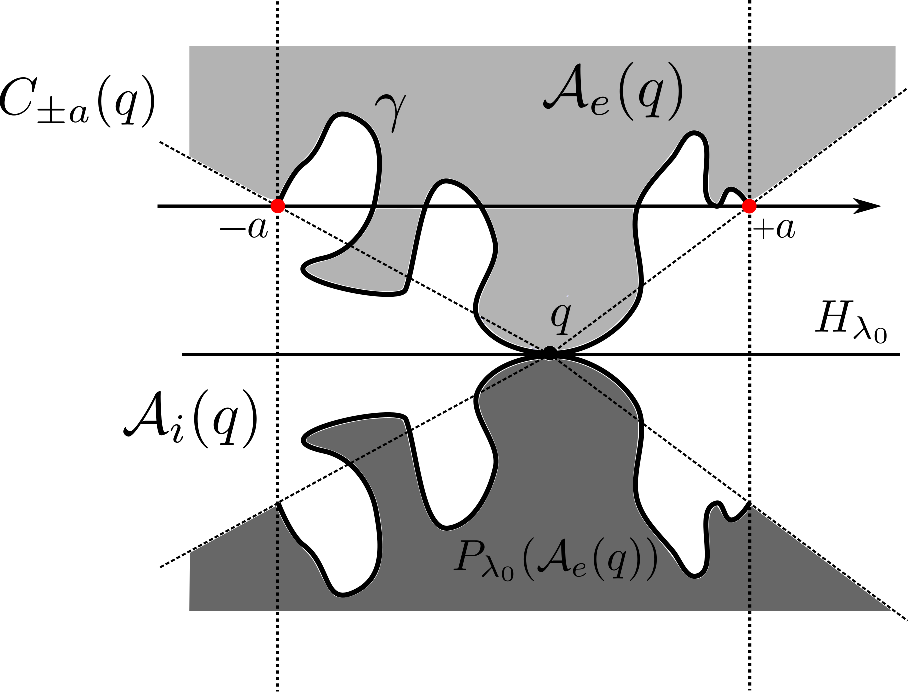}
	\end{center}
	\caption{The same situation as in Figure \ref{figure1}. The reflection $P_{\lambda_0}(\capA_e(q))$ of $\capA_e(q)$ is shown in dark gray, the set $\capA_e(q)$ in light gray.}
	\label{figure2}
\end{figure}

Moreover, we have that the volume of the set $\capA_i(q) \setminus  P_{\lambda_0}(\capA_e(q))$ is not zero because
\begin{equation*}
	\capA_i(q) \setminus  P_{\lambda_0}(\capA_e(q)) \supset \Omega^c \cap H_{\lambda_0} \cap C_{\Gamma_0}(q)^c \supset B_{\frac{|\lambda_0|}{100}}(p),
\end{equation*}
where $p= (p', \, \lambda_0) \in \mathR^{N-1} \times \mathR$ and some $p' \in \Omega^c$ with $|p'| > |\lambda_0| + a$. See also Figure \ref{figure2} for illustration in dimension 2. From \eqref{computationFracMC} and \eqref{symmetryContribution}, we obtain
\begin{align*}
	0 &= \int_{\capA_e(q)} \frac{-1}{|y-q|^{N+s}}dy + \int_{ \capA_e(q) }\frac{1}{|y-q|^{N+s}}dy + \int_{\capA_i(q) \setminus  P_{\lambda_0}(\capA_e(q))} \frac{1}{|y-q|^{N+s}}dy \nonumber\\
	&=  \int_{\capA_i(q) \setminus  P_{\lambda_0}(\capA_e(q))} \frac{1}{|y-q|^{N+s}}dy > 0,
\end{align*}
which is a contradiction.
	
\end{proof}

%\begin{remark}
%	In Figure \ref{figure1}, the critical point $\gamma$ is described as a curve contained in the cylinder $\{(x',\,x_N) \mid |x_N| < a\}$; however, our proof of Theorem \ref{theoremMinimalityHyperplane} is also valid even if the curve $\gamma$ is not necessarily contained in the cylinder. See Figure \ref{figure3} for the situation in dimension 2. 
%	

%\end{remark}

\section{Shape of Critical Points with Two Disjoint Compact Boundaries}\label{sectionProofTheorem2}
In this section, we first give the proof of Theorem \ref{theoremTwoHyperplanesNotMinimal} and then we further show some properties of critical points of $\Pers$ and compute the fractional mean curvature of cones.

\subsection{Proof of Theorem \ref{theoremTwoHyperplanesNotMinimal}}
%\{(x',\,x_N) \mid x' \in \partial \capC \}
%\{(x',\,x_N) \mid |x'| < a \}
In this subsection, we prove Theorem \ref{theoremTwoHyperplanesNotMinimal}. The idea of the proof is basically the same as the one in the proof of Theorem \ref{theoremMinimalityHyperplane}. The convexity assumption on $\capC$ is necessary for us to use the sliding method.

\begin{proof}[Proof of Theorem \ref{theoremTwoHyperplanesNotMinimal}]

We first define 
\begin{equation*}
	H_{\Gamma_i}^{+} \coloneqq \{(x',\,x_N) \mid x_N > h_i \}, \quad H_{\Gamma_i}^{-} \coloneqq \{(x',\,x_N) \mid x_N < h_i \}
\end{equation*}
for each $i \in \{1\,2\}$. Notice that 
\begin{equation*}
	\partial H_{\Gamma_i}^{+} \cap \partial \capC =  \partial H_{\Gamma_i}^{-} \cap \partial \capC  = \Gamma_i \quad \text{and} \quad \partial H_{\Gamma_i}^{+} \cap \capC = \partial H_{\Gamma_i}^{-} \cap \capC  = C_i
\end{equation*}
for each $i \in \{1,\,2\}$. 

\begin{figure}[h]
	\begin{center}
		\includegraphics[keepaspectratio, scale=0.50]{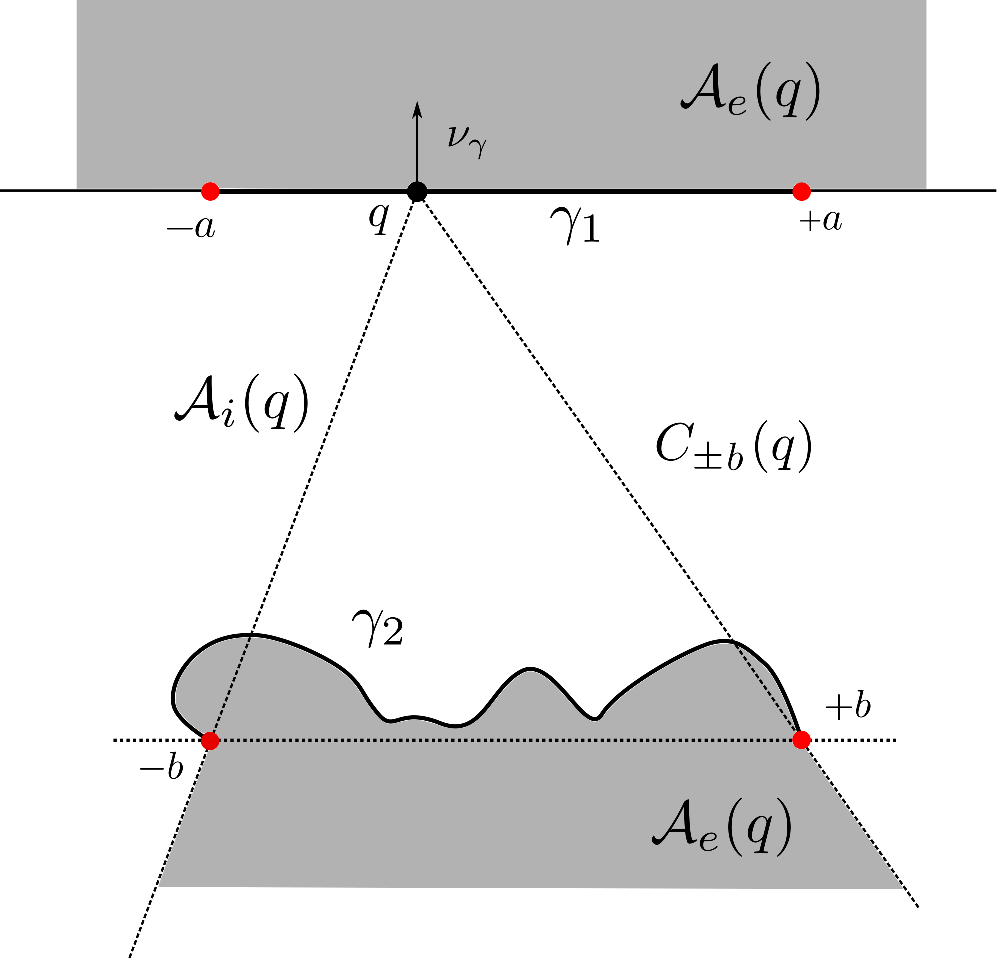}
	\end{center}
	\caption{The situation in dimension 2 in which each component $\capM_i = \gamma_i$ of the critical point $\capM = \gamma$ for $i \in \{1,\,2\}$ is a $C^{1,\alpha}$ curve with $\partial \gamma_i = \Gamma_i$ where $\Gamma_1=\{\pm a\}$ and $\Gamma_2 = \{\pm b\}$. The set $\capA_e(q)$ is shown in gray, the set $\capA_i(q)$ in white.}
	\label{figure4}
\end{figure}

Let $\capM \subset \mathR^N$ be the critical point chosen in Theorem \ref{theoremTwoHyperplanesNotMinimal}. By using the same argument as we show in the proof of Theorem \ref{theoremMinimalityHyperplane}, we obtain that $\capM$ cannot exist in the regions $H_{\Gamma_2}^{-}$ and $H_{\Gamma_1}^{+}$, that is, $\capM \cap (H_{\Gamma_2}^{-} \cup H_{\Gamma_1}^{+}) = \emptyset$.

We now show that any connected component of $\capM$ cannot be either $C_1$ or $C_2$. To see this, we assume by contradiction that there exists a connected component $\capM_1$ of $\capM$ such that $\capM_1$ coincides with $C_1$. Taking any $q \in \capM_1$, we have that the cone $C_{q, \Gamma_2}$ of vertex $q$ whose boundary passes through $\Gamma_2$ is contained in $H_{\Gamma_1}^{-}$. By choosing a proper orientation of $\capM$, we can have that $H_{\Gamma_1}^{+} \subset \capA_e(q)$ and $\capA_i(q) \subset H_{\Gamma_1}^{-}$ where the sets $\capA_e(q)$ and $\capA_i(q)$ are defined as in \eqref{interior} and \eqref{exterior}, respectively. See Figure \ref{figure4} for the situation in dimension 2. 

Since $\capM$ is a critical point of $\Pers$, from \eqref{eulerLagrangeEqPPgS}, we have that
\begin{equation}\label{eulerLagrangeMinimality2}
	0= H_{\capM,s}(q) = c_N \int_{\mathR^N} \frac{\chi_{\capA_i(q)}(y) - \chi_{\capA_e(q)}(y) }{|y-q|^{N+s}} \,dy.
\end{equation}
Now, by employing the same argument we show in the proof of Theorem \ref{theoremMinimalityHyperplane}, we obtain that
\begin{align*}
	c_N^{-1} H_{\capM,s}(q) &= \int_{\capA_e(q) \cap H_{\Gamma_1}^{+}} \frac{-1}{|y-q|^{N+s}} \,dy + \int_{\capA_e(q) \cap H_{\Gamma_1}^{-}} \frac{-1}{|y-q|^{N+s}} \,dy \nonumber\\
	&\qquad  + \int_{\capA_i(q)} \frac{1}{|y-q|^{N+s}} \,dy \nonumber\\
	&\leq \int_{B_{1/2}(-\lambda e_N)} \frac{-1}{|y-q|^{N+s}} \,dy <0
\end{align*}
because $B_{1/2}(-\lambda e_N) \subset \capA_e(q) \cap H_{\Gamma_2}^{-}$ where $\lambda > \max\{|x-z| \mid x \in C_2,\, z \in \capM\} + 1$.
%\begin{equation*}
%	.
%\end{equation*}
This contradicts \eqref{eulerLagrangeMinimality2}. Therefore, we conclude that the first claim is valid.

To prove the rest of the claim, we can argue in the same way as in the proof of the first claim. Indeed, we slide any hyperplane parallel to the $x_N$-axis from right to left or from left to right until it touches the boundary of the cylinder $\capC$. If there is no touching point, from the convexity of $\capC$, we obtain that the critical point $\capM$ is strictly contained in $\capC $ except for its boundary. Thus, we assume by contradiction that there exists a touching point $q$ of $\capM$ in the complement of $\overline{\capC}$. Then, by choosing a proper orientation of $\capM$, we can show that the contribution from $\capA_e(q)$ relative to the touching point $q$ is strictly larger (or smaller) than that from $\capA_i(q)$, respectively, as we see in the proof of the first claim. This contradicts that the fractional mean curvature vanishes at the touching point $q$. Therefore, we conclude the proof of Theorem \ref{theoremTwoHyperplanesNotMinimal}.

\end{proof}

%\section{Shape of Critical Points in $\mathR^2$}
\subsection{Further Study on Critical Points and Cones}\label{subsectionFurtherStudy}

In this subsection, we study more the shape of critical points of $\Pers$ in the same situation as in Theorem \ref{theoremTwoHyperplanesNotMinimal} with $h_1= 1$ and $h_2=-1$. 

First, we investigate the shape of critical points in dimension 2. To see this, we divide $\mathR^2$ into four regions, that is, we define four regions $C_0^{t}$, $C_0^{b}$, $C_0^{r}$, and $C_0^{\ell}$ by
\begin{align*}
	C_0^{t} &\coloneqq \{(x_1,\, x_2) \in \mathR^2 \mid  x_2 > |x_1| \}, \nonumber\\
	C_0^{b} &\coloneqq \{(x_1,\, x_2) \in \mathR^2 \mid  x_2 < -|x_1|\}, \nonumber\\
	C_0^{r} &\coloneqq \{(x_1,\, x_2) \in \mathR^2 \mid -|x_1| < x_2 < |x_1|,\, 0< x_1 \}, \nonumber\\
	\text{and} \quad C_0^{\ell} & \coloneqq \{(x_1,\, x_2) \in \mathR^2 \mid  -|x_1| < x_2 < |x_1|,\, x_1 < 0 \},
\end{align*}
respectively. Moreover, we set
\begin{equation}\label{rightCone2D}
	C_0 \coloneqq (\partial C_0^{t} \cup \partial C_0^{b}) \cap \{(x_1,\,x_2) \mid |x_2|\leq 1\}.
\end{equation}
Notice that $\partial C_0 = \Gamma_1 \cup \Gamma_2$ where $\Gamma_1$ and $\Gamma_2$ are given in Theorem \ref{theoremTwoHyperplanesNotMinimal} with $h_1= 1$ and $h_2=-1$ in $\mathR^2$. From the definition of $\Gamma_1$ and $\Gamma_2$, we have that $\Gamma_1 = \{(\pm 1, \,1)\}$ and $\Gamma_2 = \{(\pm 1, \,-1)\}$. 

Now we prove that the fractional mean curvature of the cone $C_0$ vanishes at regular points, i.e.,
\begin{equation}\label{vanishingNonlocalMeanCurvCone}
	H_{C_0,s}(z) = 0 
\end{equation}
for any $z \in C_0 \setminus \partial C_d$ with $z \neq 0$. Indeed, let $z \in C_0 \setminus \{0, (\pm 1,\,1), (\pm 1,\,-1)\}$ and, by symmetry, we may assume that $z=(z_1,\,z_2)$ satisfies $-1 < z_1 <0$ and $0< z_2 <1$. Then, from the definition of the ``interior'' $\capA_i(z)$ and the ``exterior'' $\capA_e(z)$ of the cone $C_0$ and by taking a suitable orientation of $C_0 \setminus \{0\}$, we may obtain that
\begin{equation*}
	\capA_i(z) = \left( ([z,\,(1,\,1)]^{-} \cap [z,\,(1,\,-1)]^{+}) \setminus \overline{C_0^{t}} \right) \cup \left( ( [z,\,(-1,\,1)]^{-} \cap [z,\,(-1,\,-1)]^{+} ) \cup C_0^{\ell} \right)
\end{equation*}
and
\begin{equation*}
	\capA_e(z) = \left( ([z,\,(-1,\,1)]^{+} \cap [z,\,(1,\,1)]^{+}) \cup C_0^{t} \right) \cup \left(  ( [z,\,(-1,\,-1)]^{-} \cap [z,\,(1,\,-1)]^{-} ) \setminus \overline{C_0^{\ell}} \right) 
\end{equation*}
where we denote by $[p,\,q]$ the straight line passing through $p,\,q \in \mathR^2$ with $p\neq q$ and we define $[p,\,q]^{+}$ and $[p,\,q]^{-}$ by the upper part and the lower part of the region separated by the straight line $[p,\,q]$, respectively.  

Now, because of the symmetry of the cone $C_0$, we readily observe that, in dimension 2, the sets $\capA_i(z)$ and $\capA_e(z)$ are equivalent to each other in the sense that $T(\capA_i(z)) = \capA_e(z)$ where $T : \mathR^2 \to \mathR^2$ is an isometric map such that $\frac{x+T(x)}{2} \in \{(x_1,\,x_2) \mid x_2 = x_1\}$ for any $x \in \mathR^2$. By definition, we notice that $T(z) = z$.

Therefore, from the change of variables $x \mapsto T(x)$ and , we obtain that
\begin{align*}
	c_N^{-1} H_{C_0,s}(z) &= \int_{\capA_i(z)}\frac{1}{|y-z|^{2+s}} \,dy - \int_{\capA_e(z)}\frac{1}{|y-z|^{2+s}} \,dy \nonumber\\
	&=  \int_{\capA_i(z)}\frac{1}{|y-z|^{2+s}} \,dy - \int_{\capA_i(z)}\frac{1}{|T(y)-T(z)|^{2+s}} \,dy \nonumber\\
	&= 0.
\end{align*}
By combining this fact with Theorem \ref{theoremTwoHyperplanesNotMinimal}, we can prove the following proposition.
\begin{proposition}\label{theoremCriticalPtIntersectingCone}
	Let $N=2$ and $s \in (0,\,1)$. Let $\Gamma_1$ and $\Gamma_2$ be as in Theorem \ref{theoremTwoHyperplanesNotMinimal} with $h_1= 1$ and $h_2=-1$. Let $\gamma \subset \mathR^2$ be an orientable compact $C^{1,\alpha}$ curve with $\partial \gamma = \Gamma_1 \cup \Gamma_2$. Assume that $\capC = \{(x_1,\,x_2) \mid |x_1| < 1\}$ where $\capC$ is as in \eqref{boundaryCylinder}. If $\gamma$ is a critical point of $\Pers$ under normal variations, then $\gamma$ is not contained in either $\overline{C_0^{t} \cup C_0^{b}}$ or $\overline{C_0^{r} \cup C_0^{\ell}}$ whenever $(\gamma \setminus \partial \gamma) \cap (C_0 \setminus \{0\})\neq \emptyset$.

\end{proposition}
\begin{remark}
	We may observe, by combining Proposition \ref{theoremCriticalPtIntersectingCone} with Theorem \ref{theoremTwoHyperplanesNotMinimal}, that the possible shape of minimizers of $\Pers$ in dimension 2 whose boundary is $\Gamma_1 \cup \Gamma_2$ is depicted in Figure \ref{figure6}.
\end{remark}
%and Lemma \ref{lemmaSymmetryMinimizers} in Appendix \ref{sectionSymmMini}

\begin{figure}[h]
	\begin{tabular}{cc}

	\begin{minipage}{0.48\hsize}
		\centering
		\includegraphics[keepaspectratio, scale=0.40]{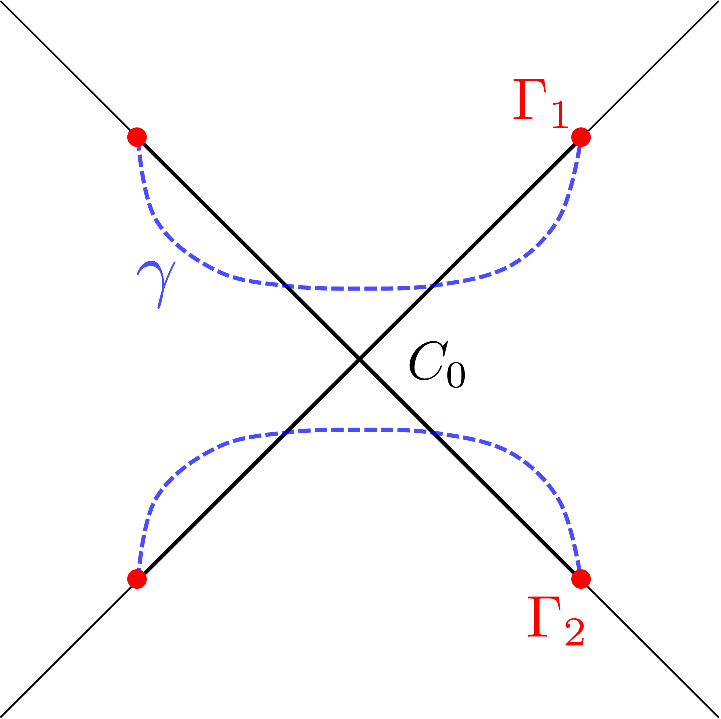}
	\end{minipage}
	
	\begin{minipage}{0.48\hsize}
		\centering
		\includegraphics[keepaspectratio, scale=0.40]{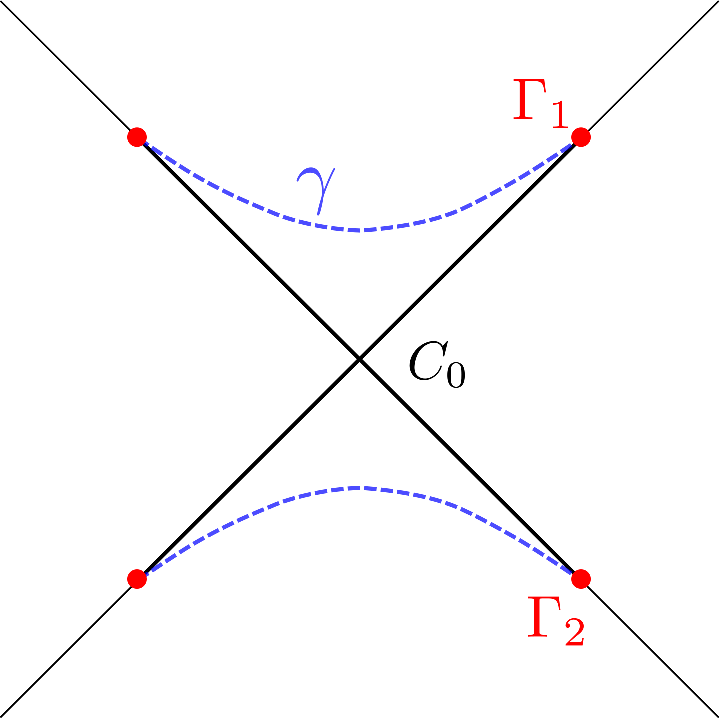}
	\end{minipage}
	
	\end{tabular}
	\caption{Possible minimizers $\gamma$ of $\Pers$ in dimension 2 with $\partial \gamma = \Gamma_1 \cup \Gamma_2$ is shown with dashed lines. On the right, $\gamma$ does not intersect with $C_0$ except at their boundaries $\Gamma_1$ and $\Gamma_2$. }
	\label{figure6}
\end{figure} 

\begin{proof}
	Let $\gamma \subset \gamma$ be as in Proposition \ref{theoremCriticalPtIntersectingCone} and we assume that $(\gamma \setminus \partial \gamma) \cap (C_0 \setminus \{0\}) \neq \emptyset$. We argue by contradiction that either $\gamma \subset  \overline{C_0^{t} \cup C_0^{b}}$ or $\gamma \subset  \overline{C_0^{r} \cup C_0^{\ell}}$ holds. Due to the symmetry of $C_0$, it is sufficient to consider the case that $\gamma \subset  \overline{C_0^{t} \cup C_0^{b}}$ holds. From the assumption, we can choose a point $z \in (\gamma \setminus \partial \gamma) \cap (C_0 \setminus \{0\})$.
	
	Now, by choosing a proper orientation, we consider the ``interior'' and ``exterior'' of $\gamma$ and $C_0$ at the touching point $z$. To see this, we set the interior and exterior at $q \in \eta$ of a curve $\eta \subset \mathR^2$ as $\capA_i^{\eta}(z)$ and $\capA_e^{\eta}(z)$, respectively. Then, from the smoothness of the critical point $\gamma$ and the assumption that $\gamma \subset  C_0^{t} \cup C_0^{b}$, we obtain, by taking a suitable orientation of $\gamma$ and $C_0$, that 
	\begin{equation}\label{nonEmptyDifferenceGammaCone01}
		|\capA_e^{C_0}(z) \setminus \capA_e^{\gamma}(z)| = |\capA_i^{\gamma}(z) \setminus \capA_i^{C_0}(z)| \neq 0
	\end{equation}
	and 
	\begin{equation}\label{nonEmptyDifferenceGammaCone02}
		|\capA_e^{\gamma}(z) \setminus \capA_e^{C_0}(z)| = |\capA_i^{C_0}(z) \setminus \capA_i^{\gamma}(z)| = 0.
	\end{equation}
	Here, from Theorem \ref{theoremTwoHyperplanesNotMinimal}, we have used the fact that all the critical points of $\Pers$ in our situation are contained in the box $\{(x_1,\,x_2) \mid |x_1| < 1, \, |x_2| < 1 \}$. See also Figure \ref{figure5} for our situation.
	
	\begin{figure}[h]
		\begin{center}
			\includegraphics[keepaspectratio, scale=0.55]{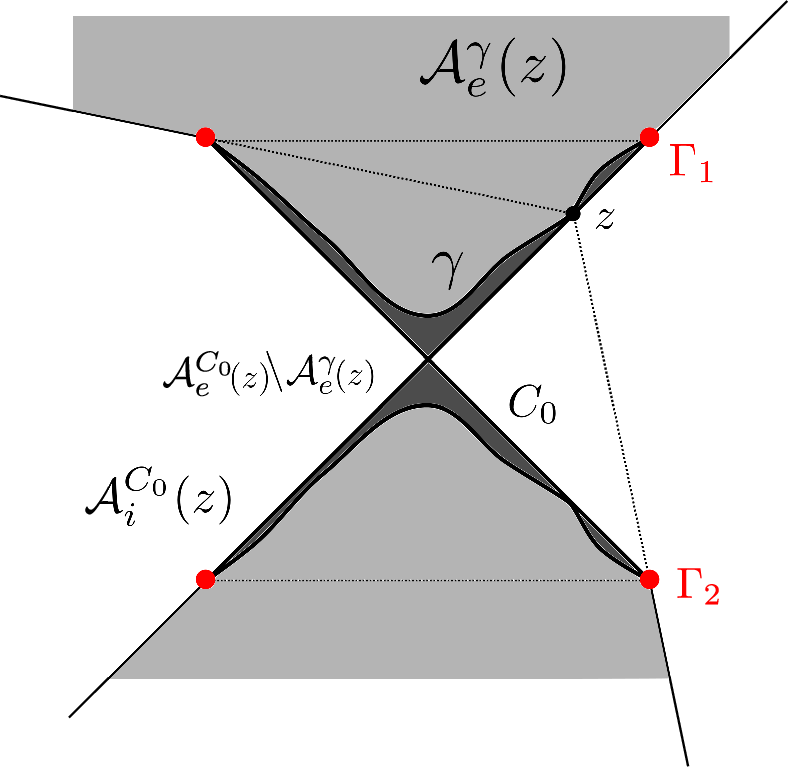}
		\end{center}
		\caption{The situation of the critical point $\gamma$ and the touching point $z$ in which $\gamma$ is included in $C_0^{t} \cup C_0^{b}$ with $\partial \gamma = \Gamma_1 \cup \Gamma_2$. The set $\capA_e^{\gamma}(z)$ is shown in light gray, the set $\capA_i^{C_0}(z)$ in white, and the set $\capA_e^{C_0}(z) \setminus \capA_e^{\gamma}(z)$ in dark gray.}
		\label{figure5}
	\end{figure}
	
	Hence, since $\gamma$ is a critical point of $\Pers$, we have that
	\begin{equation*}
		H_{\gamma,s}(z) = 0.
	\end{equation*}
	From \eqref{vanishingNonlocalMeanCurvCone}, \eqref{nonEmptyDifferenceGammaCone01}, and \eqref{nonEmptyDifferenceGammaCone02}, we have
	\begin{align}
		0 = H_{\gamma,s}(z) &= H_{\gamma,s}(z) - H_{C_0,s}(z) \nonumber\\
		&= \int_{\mathR^2} \frac{\chi_{\capA_i^{\gamma}(z)}(y) - \chi_{\capA_i^{C_0}(z)}(y) + \chi_{\capA_e^{C_0}(z)}(y) - \chi_{\capA_e^{\gamma}(z)}(y) }{|y-z|^{2+s}} \,dy \nonumber\\
		&= \int_{ \capA_i^{\gamma}(z) \setminus \capA_i^{C_0}(z) } \frac{1}{|y-z|^{2+s}} \,dy - \int_{ \capA_i^{C_0}(z) \setminus \capA_i^{\gamma}(z) } \frac{1}{|y-z|^{2+s}} \,dy \nonumber\\
		&\qquad + \int_{ \capA_e^{C_0}(z) \setminus \capA_e^{\gamma}(z) } \frac{1}{|y-z|^{2+s}} \,dy - \int_{ \capA_e^{\gamma}(z) \setminus \capA_e^{C_0}(z) } \frac{1}{|y-z|^{2+s}} \,dy \nonumber\\
		&= \int_{ \capA_i^{\gamma}(z) \setminus \capA_i^{C_0}(z) } \frac{1}{|y-z|^{2+s}} \,dy + \int_{ \capA_e^{C_0}(z) \setminus \capA_e^{\gamma}(z) } \frac{1}{|y-z|^{2+s}} \,dy > 0,
	\end{align}
	 which is a contradiction. Therefore we obtain the claim.
\end{proof}

\begin{remark}\label{remarkCriticalPts2D}
	We briefly consider the situation of Theorem \ref{theoremTwoHyperplanesNotMinimal} with $h_1 = d$ and $h_2=-d$ for $d \neq 1$ and $d>0$ and see what kind of shape the critical points in dimension 2 look like. Notice that we have treated the case of $d=1$ in Proposition \ref{theoremCriticalPtIntersectingCone}.
	
	Assume that $h_1=d$ and $h_2=-d$ for $d>0$. We define a cone $C_d$ of vertex $0$ by
	\begin{equation}\label{cone2DPositiveorNegativeFracMC}
		C_d \coloneqq \{(x_1,\,x_2) \in \mathR^2 \mid |x_2|=d|x_1|, \, |x_2| \leq d\}.
	\end{equation}
	Notice that $\partial C_d = \Gamma_1 \cup \Gamma_2$. By slightly modifying the argument for showing that $H_{C_0,s}=0$ on $C_0 \setminus (\partial C_0 \cup \{0\})$ and taking a proper orientation, we can show that the fractional mean curvature $H_{C_d,s}(z)$ of $C_d$ is either positive or negative for any $z \in C_d \setminus \partial C_d$ with $z \neq 0$. Then, again by slightly modifying the argument in the proof of Proposition \ref{theoremCriticalPtIntersectingCone}, we obtain the same result as in Proposition \ref{theoremCriticalPtIntersectingCone} even for any $d \neq 1$.
\end{remark}

We next prove the same result as Proposition \ref{theoremCriticalPtIntersectingCone} in higher dimensions. To see this, we also show that the fractional mean curvature of a cone passing through $\Gamma_1 \cup \Gamma_2$ is either positive or negative everywhere except at its vertex in higher dimensions. The idea of the proof is the same as that in the proof of Proposition \ref{theoremCriticalPtIntersectingCone}. We first give some notations. We define a bounded tube $D_0$ and a unbounded cone $\widetilde{C}_0$ by
\begin{align*}
	D_0 &\coloneqq \{(x',\,x_N) \in \mathR^{N-1} \times \mathR \mid |x'| < 1, \, -1 < x_N < 1\} \nonumber \\
	\widetilde{C}_0 &\coloneqq \{(x',\,x_N) \in \mathR^{N-1} \times \mathR \mid |x_N| > |x'| \}.
\end{align*}
Moreover, we set $C^N_0 \coloneqq \partial \widetilde{C}_0 \cap \{(x',\,x_N) \mid |x_N| \leq 1\}$ and decompose $\widetilde{C}_0$ into two parts $\widetilde{C}_0^{+}$ and $\widetilde{C}_0^{-}$ which are defined by
\begin{align*}
	\widetilde{C}_0^{+} &\coloneqq \{(x',\,x_N) \in \mathR^{N-1} \times \mathR \mid x_N > |x'| \} \nonumber\\
	\widetilde{C}_0^{-} &\coloneqq \{(x',\,x_N) \in \mathR^{N-1} \times \mathR \mid x_N < -|x'| \}.
\end{align*}
Notice that $C^N_0$ coincides with $C_0$ given in \eqref{rightCone2D} if $N=2$ and $\partial C^N_0 = \Gamma_1 \cup \Gamma_2$.
\begin{proposition}\label{propositionCriticalIntersectingConeN>3}
	Let $N \geq 3$ and $s \in (0,\,1)$. Let $\Gamma_1$ and $\Gamma_2$ be as in Theorem \ref{theoremTwoHyperplanesNotMinimal} with $h_1= 1$ and $h_2=-1$. Let $\capM \subset \mathR^N$ be an orientable compact $C^{1,\alpha}$ manifold with $\partial \capM = \Gamma_1 \cup \Gamma_2$. Assume that $\capC = \{(x',\,x_N) \mid |x'| < 1\}$ where $\capC$ is as in \eqref{boundaryCylinder}. If $\capM$ is a critical point of $\Pers$ under normal variations, then $\capM$ is not contained in either $\overline{\widetilde{C}_0^{+} \cup \widetilde{C}_0^{-}}$ or $D_0 \setminus (\overline{\widetilde{C}_0^{+} \cup \widetilde{C}_0^{-}})$ whenever $(\capM \setminus \partial \capM) \cap (C^N_0 \setminus \{0\})\neq \emptyset$.
\end{proposition} 
\begin{proof}
	The proof is similar to that of Proposition \ref{theoremCriticalPtIntersectingCone} and we here show a rough sketch of the proof. Let $\capM$ be the critical point selected in Proposition \ref{propositionCriticalIntersectingConeN>3}. We assume that $(\capM \setminus \partial \capM) \cap (C^N_0 \setminus \{0\})$ is not empty and we choose a point $z \in (\capM \setminus \partial \capM) \cap (C^N_0 \setminus \{0\})$. Suppose by contradiction that either 
	\begin{equation*}
		\capM \subset \overline{\widetilde{C}_0^{+} \cup \widetilde{C}_0^{-}} \quad \text{or} \quad \capM \subset D_0 \setminus (\overline{\widetilde{C}_0^{+} \cup \widetilde{C}_0^{-}})
	\end{equation*}
	holds. First, by choosing a proper orientation, we show that 
	\begin{equation}\label{fracMCConeatZ}
		H_{C^N_0, s}(z) > 0.
	\end{equation}
	Indeed, if we take the unit normal vector $\nu_{C^N_0}(z)$ of the cone $C^N_0$ at $z$ in such a way that the direction is towards $\widetilde{C}_0$, then the ``interior'' $\capA_i^{C^N_0}(z)$ and ``exterior'' $\capA_e^{C^N_0}(z)$ can be defined as
	\begin{equation*}
		\capA_i^{C^N_0}(z) = \mathR^N \setminus \left( \capA_e^{C^N_0}(z) \cup C^N_0 \right)
	\end{equation*}
	and 
	\begin{equation*}
		\capA_e^{C^N_0}(z) = (\widetilde{C}_0 \cap \{(x',\,x_N) \mid |x_N| \leq 1\}) \cup  \left( (C_{\Gamma_1}(z) \cup C_{\Gamma_2}(z)) \cap \{(x',\,x_N) \mid |x_N| \geq 1\} \right) 
	\end{equation*}
	where $C_{\Gamma_i}(z)$ is defined by a (filled) cone of vertex $z$ passing through $\Gamma_i$ for each $i \in \{1,2\}$. Now we take a hyperplane $H_z$ which is tangent to $\partial \widetilde{C}_0$ and passes through $z$ and define the reflection map $T_{H_z}$ with respect to $H_z$. From the definitions of $C^N_0$, $\capA_i^{C^N_0}(z)$, and $\capA_e^{C^N_0}(z)$, we have
	\begin{equation*}
		T_{H_z}(\capA_e^{C^N_0}(z)) \subset \capA_i^{C^N_0}(z) \quad \text{and} \quad \left|\capA_i^{C^N_0}(z) \setminus T_{H_z}(\capA_e^{C^N_0}(z))\right| \neq 0.
	\end{equation*}
	Since $T_{H_z}$ is an isometry and $T_{H_z}(z) = z$, we obtain the following:
	\begin{align}
		H_{C^N_0, s}(z) &= \int_{\capA_i^{C^N_0}(z) \setminus T_{H_z}(\capA_e^{C^N_0}(z))} \frac{dx}{|x-z|^{N+s}}+ \int_{T_{H_z}(\capA_e^{C^N_0}(z))} \frac{dx}{|x-z|^{N+s}} \nonumber\\
		&\qquad - \int_{\capA_e^{C^N_0}(z)} \frac{dx}{|x-z|^{N+s}} \nonumber\\
		&= \int_{\capA_i^{C^N_0}(z) \setminus T_{H_z}(\capA_e^{C^N_0}(z))} \frac{dx}{|x-z|^{N+s}} + 0 > 0,
	\end{align}
	which implies \eqref{fracMCConeatZ}.
	
	Now, since $\capM$ is a critical point of $\Pers$, we have the Euler-Lagrange equation
	\begin{equation*}
		H_{\capM,s}(z) = 0.
	\end{equation*}
	Thus, taking the unit normal vector $\nu_{\capM}(z)$ of $\capM$ at $z$ as $\nu_{C^N_0}(z)$, we can have the following computation:
	\begin{align}
		0 &= H_{\capM,s}(z) - H_{C^N_0,s}(z) + H_{C^N_0,s}(z) \nonumber \\
		&= 2\int_{\capA_e^{C^N_0}(z) \setminus \capA_e^{\capM}(z)} \frac{1}{|x-z|^{N+s}}\,dx - 2\int_{\capA_e^{\capM}(z) \setminus \capA_e^{C^N_0}(z)} \frac{1}{|x-z|^{N+s}}\,dx + H_{C^N_0,s}(z) \label{computationFracMCConeCritPt}.
	\end{align}
	From the assumption, we can observe that
	\begin{equation*}
		\left| \capA_e^{C^N_0}(z) \setminus \capA_e^{\capM}(z) \right| > 0\quad \text{and} \quad \left| \capA_e^{\capM}(z) \setminus \capA_e^{C^N_0}(z) \right| = 0.
	\end{equation*}
	Therefore, from \eqref{fracMCConeatZ} and \eqref{computationFracMCConeCritPt}, we reach a contradiction.
\end{proof}

\section{Topology of Critical Points}\label{sectionTheorem3}
In this section, we investigate the topology of critical points with two parallel and co-axial boundaries and prove Theorem \ref{theoremConnectednessCriticalPtsN3} and \ref{theoremDisconnectednessCriticalPts}. 

%The idea for proving the latter claim in Theorem \ref{theoremPoppingCriticalPtwithTwoBdry} is as follows: first, we construct a barrier against the critical point in such a way that the fractional mean curvature is negative on the barrier for large $d$. The boundary of the barrier is same as that of the critical point and the barrier contains a small ``bump'' whose height depending on a small parameter $\varepsilon$. Now we vary the parameter $\varepsilon$ starting at 0 to some point until it touches the critical point. Then we obtain the Euler-Lagrange equation for the critical point at the touching point. From our construction of the barrier, it is somehow a ``sub-solution'' or ``super-solution'' of the Euler-Lagrange equation. In this situation, the critical point cannot share the touching point with the barrier due to the ``maximal principle'' as one can see in the classical mean curvature equation.
%\begin{align*}
%	
%	\begin{cases}
%		< 0 &\quad \text{if $d <1$,} \\
%		> 0 &\quad \text{if $d>1$} 
%	\end{cases}
%\end{align*} 

Before proving our main theorems of this section, we show Lemma \ref{theoremPoppingCriticalPtwithTwoBdry}. The idea of the proof is to construct a small barrier, whose fractional mean curvature is strictly positive or negative, and to ``slide'' the barrier until it touches the critical point. The construction of the barrier is inspired by the one shown in \cite{DSV23Bdry}. See also \cite[Proof of Proposition 4.1]{DOV22}. In the sequel, without loss of generality, we may assume that $\capC = \{(x',\,x_N) \mid |x'| < 1\}$ where $\capC$ is as in \eqref{boundaryCylinder} for simplicity.

\begin{proof}[Proof of Lemma \ref{theoremPoppingCriticalPtwithTwoBdry}]
 	We first fix $\varepsilon \in (0,\,1)$ so small that $\delta=\delta(\varepsilon) \coloneqq (-\log \varepsilon)^{-1/2} < \frac{1}{2}$ and we define a smooth bump function $w_{\varepsilon}: \mathR^{N-1} \to \mathR$ by 
	\begin{align*}
		w_{\varepsilon}(x') \coloneqq 
		\begin{cases}
			-\exp\left( - \frac{1}{\delta^2-|x'|^2} \right) & \quad \text{for $|x'| < \delta$} \\ 
			0 & \quad \text{otherwise}.
		\end{cases}
	\end{align*}
	Notice that $w_{\varepsilon} \in C^{\infty}(\mathR^{N-1})$, $w_{\varepsilon}(x')=0$ for $|x'|=\delta$, $w_{\varepsilon}(0) = -\varepsilon$, and 
	\begin{equation}\label{auxilliaryFuncPhi}
		\lim_{\varepsilon \downarrow 0}\phi(\varepsilon) \coloneqq  \lim_{\varepsilon \downarrow 0}\|\nabla'^2 w_{\varepsilon}\|_{C^0} = 0.
	\end{equation}
	If necessary, we may choose $\varepsilon$ in such a way that $\phi(\varepsilon)<1$. Note that, since $\phi$ is an increasing function in a neighborhood $I_{\phi} \subset [0,\,1)$ of the origin, its inverse function $\phi^{-1}$ exists in a neighborhood $J_{\phi} \subset [0,\,1)$ of the origin. We then set
	\begin{equation}\label{radiausFuncEpsilon}
		r(\varepsilon) \coloneqq \left(2(N-1)\phi(\varepsilon)\right)^{-1} \quad \text{and} \quad d(\varepsilon) \coloneqq 2r(\varepsilon).
	\end{equation}
	Moreover, we define a positive constant $\varepsilon_d$ as
	\begin{align*}
		\varepsilon_d \coloneqq
		\begin{cases}
			\phi^{-1}(( 2(N-1)d)^{-1}) &\quad \text{if $(2(N-1)d)^{-1} \in J_{\phi}$} \\
			(\text{any positive constant in $J_{\phi}$}) &\quad \text{if $(2(N-1)d)^{-1} \not\in J_{\phi}$}.
		\end{cases}
	\end{align*}
	By definition, we observe that $r(\varepsilon_d) \geq d$ and $\varepsilon_d$ can be chosen independently of $d$ if $d < (2(N-1))^{-1}$ since $J_{\phi} \subset [0,\,1)$.
	
	In addition, we choose a smooth function $v_{\varepsilon}:\mathR^{N-1} \to \mathR$ such that $v_{\varepsilon}$ is radially symmetric, $0 \leq v_{\varepsilon}(x') \leq 1$ for $x' \in \mathR^{N-1}$, and $\spt v_{\varepsilon} \subset B'_{1/8}(0)$ where we denote by $B'_r(0)$ an open ball centered at the origin of radius $r$ in $\mathR^{N-1}$. In particular, we choose $v_{\varepsilon}$ in such a way that its subgraph $\{(x',\,x_N) \mid 0 \leq x_N \leq v_{\varepsilon}(x')\}$ of $v_{\varepsilon}$ contains a cylinder of height $\phi(\varepsilon)^{\beta} < 1$ for $\beta \in (0,\,s)$ with the base of radius $\frac{1}{16}$. Then we define a function $\widetilde{w}_{\varepsilon} : \mathR^{N-1} \to \mathR$ by
	\begin{align*}
		\widetilde{w}_{\varepsilon}(x') \coloneqq
		\begin{cases}
			w_{\varepsilon}(x') &\quad  \text{for $|x'| < \delta$} \\
			0 &\quad \text{for $\delta \leq |x'| < \frac{5}{8}$} \\
			v_{\varepsilon}\left(x'- b'\right) &\quad \text{for $\frac{5}{8} \leq |x'| < \frac{7}{8}$} \\
			0 &\quad \text{for $|x'| \geq \frac{7}{8}$}
		\end{cases}
	\end{align*}
	where $b' \in \mathR^{N-1}$ is any point with $|b'|=\frac{3}{4}$. Notice that $\widetilde{w}_{\varepsilon}$ is smooth in $\mathR^{N-1}$.
%	Moreover, we choose $d>1$ in such a way that $0 < \varepsilon < \delta$.
%	$\phi: (0,\,1) \to (0,\,1)$ is a given increasing function such that $\phi(\varepsilon) \thickapprox \varepsilon^{\alpha_1}(-\log \varepsilon)^{\alpha_2}$ for
%	\textcolor{red}{We now take any $\varepsilon>0$ and a convex $C^2$-function $w : \mathR^{N-1} \to \mathR$ such that $w(x') = 0$ for any $|x'|=1/4$ and $\|\nabla'^2 w\|_{C^0} \leq 2\sqrt{N-1}$. Such a function exists and we can choose, for instance, $w(x') \coloneqq |x'|^2-1/4$. Then we define a function $w_{\varepsilon}: \mathR^{N-1} \to \mathR$ by $w_{\varepsilon}(x') \coloneqq \varepsilon\,w(x')$ for $x' \in \mathR^{N-1}$ and, moreover, we define a function $\widetilde{w}_{\varepsilon} : \mathR^{N-1} \to \mathR$ by
%	\begin{align*}
%		\begin{cases}
%			w_{\varepsilon}(x') &\quad  \text{for $|x'| < 1/4$} \\
%			0 &\quad \text{for $1/4 \leq |x'| < 1$} \\
%			\text{the graph function of $\capC_{w_{\varepsilon}}(q)$} &\quad \text{for $|x'| \geq 1$}.
%		\end{cases}
%	\end{align*}
%	}

	Now we construct a barrier against $\widetilde{\capM}^{\varepsilon,t}$, i.e., an orientable compact $(N-1)$-dimensional piecewise smooth manifold $\widetilde{\capM}^{\varepsilon,t}$ in the following way: first, taking any $t \in (0,\,\varepsilon]$, we define four sets
	\begin{align*}
		&\capM_1^{\varepsilon, t} \coloneqq \{(x',\,x_N) \mid |x'| \leq 1, \, x_N = \widetilde{w}_{\varepsilon}(x') + t\}, \nonumber\\
		\text{and} \quad &\capM_2^{\varepsilon,t} \coloneqq \{(x',\, x_N) \mid |x'| \leq 1, \, x_N = -d(\varepsilon) + t\}
%		&\capM_{b,1}^{\varepsilon} \coloneqq \{(x',\,x_N) \mid |x'| = 1, \, 0<x_N<t \} \nonumber\\
%		 &\capM_{b,2}^{\varepsilon} \coloneqq \{(x',\,x_N) \mid |x'|=1, \, -d(\varepsilon)+t < x_N < -d\}
	\end{align*}
	where $d(\varepsilon) \coloneqq 2r(\varepsilon)$. Then we define our barrier as $\widetilde{\capM}^{\varepsilon,t} \coloneqq \capM_1^{\varepsilon,t} \cup \capM_2^{\varepsilon,t}$. By construction, we can easily see that $\widetilde{\capM}^{\varepsilon,t}$ is an orientable compact $(N-1)$-dimensional smooth manifold with $\partial \capM_1^{\varepsilon} = \Gamma^{\varepsilon,t}_1$ and $\partial \capM_2^{\varepsilon,t}  = \Gamma^{\varepsilon,t}_2$ where we define
	\begin{equation*}
		\Gamma_1^{\varepsilon,t} \coloneqq \capC \cap \{x_N = t\} \quad \text{and} \quad \Gamma_2^{\varepsilon,t} \coloneqq \capC \cap \{x_N = -d(\varepsilon) + t\}.
	\end{equation*}
	
	\begin{figure}[h]
		\begin{center}
			\includegraphics[keepaspectratio, scale=0.60]{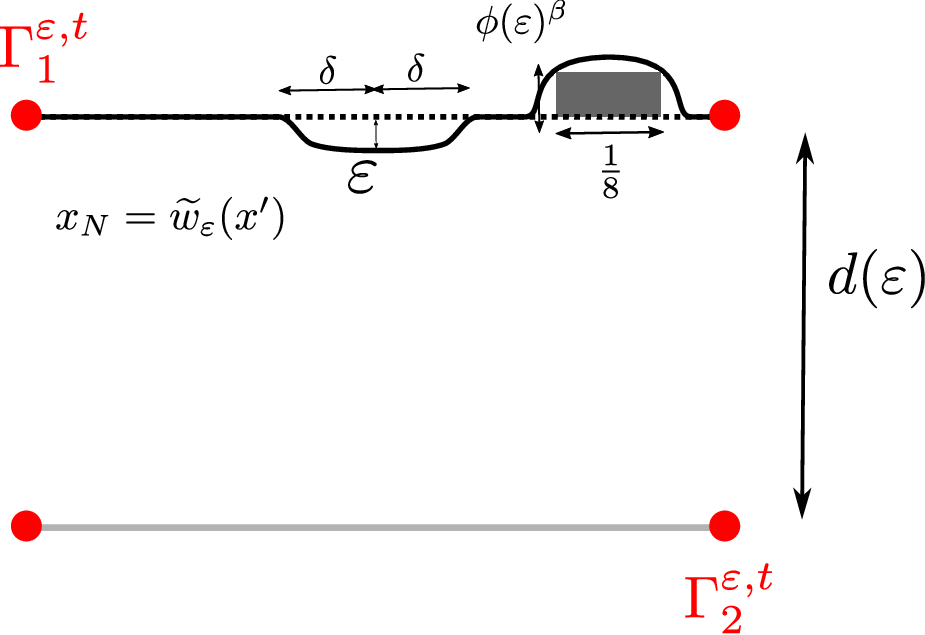}
		\end{center}
		\caption{The barrier $\widetilde{\capM}^{\varepsilon,t}= \capM^{\varepsilon,t}_1 \cup \capM_2^{\varepsilon,t}$ associated with a function $\widetilde{w}_{\varepsilon}$ in dimension 2. The graph of $\widetilde{w}_{\varepsilon}$ in $\{|x'|<1\}$ is depicted with black lines and the cylinder in dark gray.}
		\label{figureBarrierThreeBumps}
	\end{figure}
	
	We next construct another barrier in which the small bump associated with $v_{\varepsilon}$ is removed from $\widetilde{\capM}^{\varepsilon,t}$. First, for any $t \in (0,\,\varepsilon]$, we define a manifold $\capM_3^{\varepsilon,t}$ as the graph of $w_{\varepsilon}$, i.e., 
	\begin{equation*}
		\capM_3^{\varepsilon,t} \coloneqq \{(x',\,x_N) \mid |x'| < 1,\, x_N =  w_{\varepsilon}(x') + t \}
	\end{equation*}
	and, then, define the second barrier as $\capM^{\varepsilon,t} \coloneqq \capM_3^{\varepsilon,t} \cup \capM_2^{\varepsilon,t}$. Notice that $\partial \capM^{\varepsilon,t} = \Gamma_1^{\varepsilon,t} \cup \Gamma_2^{\varepsilon,t}$.

	We now show, up to orientation, that the fractional mean curvature of $\widetilde{\capM}^{\varepsilon,t}$ is negative on the graph of $w_{\varepsilon}$. Let $q \in \capM_1^{\varepsilon,t}$ be any point such that $|q'| < \delta(\varepsilon)$ where we set $q=(q',\,q_N)$. We now define $C_{\Gamma_i^{\varepsilon,t}}(q)$ by a (filled) cone of vertex $q$ whose boundary passes through $\Gamma_i^{\varepsilon,t}$ for $i \in \{1,\,2\}$. Then, up to orientation, we can choose the interior and exterior of $\widetilde{\capM}^{\varepsilon,t}$ at $q$ as 
	\begin{equation*}
		\capA^{\widetilde{\capM}^{\varepsilon,t}}_i(q) \coloneqq \mathR^N \setminus \left( \capA^{\widetilde{\capM}^{\varepsilon,t}}_e(q) \cup \widetilde{\capM}^{\varepsilon,t} \right)
	\end{equation*}
	and 
	\begin{align*}
		\capA^{\widetilde{\capM}^{\varepsilon,t}}_e(q) \coloneqq \left( C_{\Gamma_2^{\varepsilon,t}}(q) \cap \{(x',\,x_N) \mid x_N < -d(\varepsilon) + t \} \right) \cup \{(x',\,x_N) \mid x_N > \widetilde{w}^q_{\varepsilon}(x')\},
	\end{align*}
	respectively, where we define a function $\widetilde{w}^q_{\varepsilon} : \mathR^{N-1} \to \mathR$ by
	\begin{align*}
		\widetilde{w}^q_{\varepsilon}(x') \coloneqq 
		\begin{cases}
			\widetilde{w}_{\varepsilon}(x') &\quad \text{for $|x'| < 1$} \\
			(\text{the graph function of $\partial C_{\Gamma_1^{\varepsilon, t}}(q)$}) &\quad \text{for $|x'| \geq 1$}. 
		\end{cases}
	\end{align*}
%	Then, by definition, we observe that the boundary of $\capA^{\widetilde{\capM}^{\varepsilon,t}}_e(q)$ or, equivalently, $\capA^{\widetilde{\capM}^{\varepsilon,t}}_i(q)$ can be represented as the graph of $\widetilde{w}^q_{\varepsilon}$.
	
	We now compute the fractional mean curvature $H_{\widetilde{\capM}^{\varepsilon,t},s}(q)$ at $q$ of $\widetilde{\capM}^{\varepsilon,t}$. From the definition of the fractional mean curvature and by the change of variables, we have
	\begin{align}
		- H_{\widetilde{\capM}^{\varepsilon,t},s}(q) &= \int_{\mathR^N} \frac{\chi_{\capA_e^{\widetilde{\capM}^{\varepsilon,t}}(q)}(q - x) - \chi_{\capA_i^{\widetilde{\capM}^{\varepsilon,t}}(q)}(q-x) }{|x|^{N+s}} \,dx  \nonumber\\
		&=  \int_{B'_{r}(0) \times (-r,\,r)} \frac{\chi_{\capA_e^{\widetilde{\capM}^{\varepsilon,t}}(q)}(q - x) - \chi_{\capA_i^{\widetilde{\capM}^{\varepsilon,t}}(q)}(q-x) }{|x|^{N+s}} \,dx \nonumber\\
		&\qquad + \int_{(B'_{r}(0) \times (-r,\,r))^c} \frac{\chi_{\capA_e^{\widetilde{\capM}^{\varepsilon,t}}(q)}(q - x) - \chi_{\capA_i^{\widetilde{\capM}^{\varepsilon,t}}(q)}(q-x) }{|x|^{N+s}} \,dx \nonumber\\
		&\eqqcolon (I) + (II) \label{fracMeanCurvatureComp01}
	\end{align}
	where we set $r \coloneqq r(\varepsilon)$ where $r(\varepsilon)$ is as in \eqref{radiausFuncEpsilon}. 
	
	We first compute $(I)$. Thanks to the choice of $r$ and the construction of $\widetilde{\capM}^{\varepsilon,t}$, we observe that
	\begin{equation*}
		(B'_r(q) \times (-r,\,r)) \cap  \left(C_{\Gamma_2^{\varepsilon,t}}(q) \cap \{(x',\,x_N) \mid x_N < -d(\varepsilon)+t\} \right) = \emptyset.
	\end{equation*}
	Thus we can represent the set $\partial \capA_e^{\widetilde{\capM}^{\varepsilon,t}}(q)$ in $B'_r(0) \times (-r,\,r)$ as the graph of $\widetilde{w}^q_{\varepsilon}$. By doing the similar computation to the one in \cite[Section 3]{BFV14}, we obtain
	\begin{align}
		 (I) &= - 2 \int_{B'_r(0)} F\left( \frac{\widetilde{w}^q_{\varepsilon}(q') - \widetilde{w}^q_{\varepsilon}(q'-x')}{|x'|}\right) \,\frac{dx'}{|x'|^{N-1+s}}  \nonumber\\
		&= - \int_{B'_r(0)} F\left( \frac{\widetilde{w}^q_{\varepsilon}(q') - \widetilde{w}^q_{\varepsilon}(q'-x')}{|x'|}\right) \,\frac{dx'}{|x'|^{N-1+s}} \nonumber\\
		&\qquad - \int_{B'_r(0)} F\left( \frac{\widetilde{w}^q_{\varepsilon}(q') - \widetilde{w}^q_{\varepsilon}(q'+x')}{|x'|}\right) \,\frac{dx'}{|x'|^{N-1+s}} \nonumber\\
		&= \int_{B'_r(0)} F\left( \frac{-\widetilde{w}^q_{\varepsilon}(q') + \widetilde{w}^q_{\varepsilon}(q'-x')}{|x'|}\right) \,\frac{dx'}{|x'|^{N-1+s}} \nonumber\\
		&\qquad - \int_{B'_r(0)} F\left( \frac{\widetilde{w}^q_{\varepsilon}(q') - \widetilde{w}^q_{\varepsilon}(q'+x')}{|x'|}\right) \,\frac{dx'}{|x'|^{N-1+s}}
		 \label{fracMeanCurvatureComp02}
	\end{align}
%	where a function $\widetilde{w}^q_{\varepsilon} : \mathR^{N-1} \to \mathR$ is given by
%	\begin{align*}
%		\begin{cases}
%			w_{\varepsilon}(x') &\quad  \text{for $|x'| < 1$} \\
%			\text{the graph function of $\capC_{w_{\varepsilon}}(q)$} &\quad \text{for $|x'| \geq 1$}
%		\end{cases}
%	\end{align*}
	where we set
	\begin{equation*}
		F(t) \coloneqq \int_{0}^{t} \frac{1}{(1+\sigma^2)^{\frac{N+s}{2}}} \,d\sigma
	\end{equation*}
	for any $t\in\mathR$. Note that we have used the change of variables $x' \mapsto -x'$ in the second equality of \eqref{fracMeanCurvatureComp02} and the fact that $F$ is odd in the last equality of \eqref{fracMeanCurvatureComp02}. By definition, we have that $\widetilde{w}^q_{\varepsilon}(q') = w_{\varepsilon}(q')$ and $\widetilde{w}^q_{\varepsilon} \geq w_{\varepsilon}$ in $\mathR^{N-1}$. 
%	\begin{equation*}
%		 \quad \text{in }.
%	\end{equation*}
	Since $F$ is increasing, we derive from \eqref{fracMeanCurvatureComp02} that
	\begin{align}
		(I) & \geq \int_{B'_r(0)} F\left( \frac{-w_{\varepsilon}(q') + w_{\varepsilon}(q'-x')}{|x'|}\right) \,\frac{dx'}{|x'|^{N-1+s}} \nonumber\\
		&\qquad - \int_{B'_r(0)} F\left( \frac{w_{\varepsilon}(q') - w_{\varepsilon}(q'+x')}{|x'|}\right) \,\frac{dx'}{|x'|^{N-1+s}}.
	\end{align}
%	\textcolor{red}{By definition, we have that $\widetilde{w}_{\varepsilon}(q') = w_{\varepsilon}(q')$ and 
%	\begin{equation*}
%		\widetilde{w}_{\varepsilon} \leq w_{\varepsilon} \quad \text{in $\mathR^{N-1}$}.
%	\end{equation*}
%	Since $F$ is increasing, we derive from \eqref{fracMeanCurvatureComp02} that
%	\begin{align}
%		(I) & \leq \int_{B'_r(0)} F\left( \frac{-w_{\varepsilon}(q') + w_{\varepsilon}(q'-x')}{|x'|}\right) \,\frac{dx'}{|x'|^{N-1+s}} \nonumber\\
%		&\qquad - \int_{B'_r(0)} F\left( \frac{w_{\varepsilon}(q') - w_{\varepsilon}(q'+x')}{|x'|}\right) \,\frac{dx'}{|x'|^{N-1+s}}.
%	\end{align}
%	}
	Now, by using the fundamental theorem of calculus in \eqref{fracMeanCurvatureComp02}, we obtain
	\begin{equation}\label{fracMeanCurvatureComp03}
		(I)  \geq - \int_{B'_r(0)} \int_{0}^{1} F'\left( a(x',\,q',\,\lambda) \right) \,d\lambda \frac{2w_{\varepsilon}(q') - w_{\varepsilon}(q'+x') - w_{\varepsilon}(q'-x')}{|x'|^{N+s}}\,dx'
	\end{equation}
	where we set $a(x',\,q',\,\lambda)$ as
	\begin{equation*}
		 a(x',\,q',\,\lambda) \coloneqq \lambda \frac{w_{\varepsilon}(q') - w_{\varepsilon}(q'+ x')}{|x'|} + (1-\lambda)\frac{-w_{\varepsilon}(q') + w_{\varepsilon}(q'-x')}{|x'|}
	\end{equation*}
	for $x',\,q' \in \mathR^{N-1}$ and $\lambda \in [0,\,1]$. Thus, by using again the fundamental theorem of calculus, we have
	\begin{equation*}
		(I) \geq -\int_{B'_r(0)}\int_{0}^{1}\frac{|\nabla' w_{\varepsilon}(q'+ \rho x') - \nabla'w_{\varepsilon}(q' - \rho x')|}{|x'|^{N-1+s}} \, d\rho\,dx'.
	\end{equation*}
	Since $w_{\varepsilon}$ is smooth in $\mathR^{N-1}$, we then have
	\begin{equation}
		(I) \geq -2\|\nabla'^2 w_{\varepsilon}\|_{C^0} \int_{B'_r(0)} \frac{dx'}{|x'|^{N-2+s}} = -\frac{2\omega_{N-2}}{1-s} \|\nabla'^2 w_{\varepsilon}\|_{C^0} \,r^{1-s}. \label{fracMeanCurvatureComp04}
	\end{equation}
	Now we compute $(II)$ in the following way: since $B_r(0)  \subset  B'_r(0) \times (-r,\,r) \subset \mathR^N$, we have
	\begin{equation}
		(II) \geq - \int_{B^c_r(0)} \frac{dx}{|x|^{N+s}} = - \frac{\omega_{N-1}}{s} r^{-s}.\label{fracMeanCurvatureComp05}
	\end{equation}
	Therefore, from \eqref{fracMeanCurvatureComp04} and \eqref{fracMeanCurvatureComp05}, we obtain
	\begin{equation}\label{fracMeanCurvatureComp06}
		- H_{\widetilde{\capM}^{\varepsilon,t},s}(q) \geq - \left( c_1\,\|\nabla'^2 w_{\varepsilon}\|_{C^0} \,r^{1-s} + c_2 \,r^{-s} \right)
	\end{equation}
	where $c_1$ and $c_2$ are defined as
	\begin{equation*}
		c_1 \coloneqq \frac{2\omega_{N-2}}{1-s} \quad \text{and} \quad c_2 \coloneqq \frac{\omega_{N-1}}{s},
	\end{equation*}
	respectively. From \eqref{radiausFuncEpsilon}, it holds that the right-hand side of \eqref{fracMeanCurvatureComp06} takes the maximum at $r=r(\varepsilon) \in (0,\,d(\varepsilon))$. Hence we finally obtain, from \eqref{fracMeanCurvatureComp06}, that
	\begin{equation}\label{fracMeanCurvatureMBound}
		- H_{\widetilde{\capM}^{\varepsilon,t},s}(q) \geq - c \, \|\nabla'^2 w_{\varepsilon}\|_{C^0}^s = - c \, \phi(\varepsilon)^s
	\end{equation}
	where we set the constant $c= c(N,s)>0$ as
	\begin{equation*}
		c = c(N,s) \coloneqq \frac{(2(N-1))^s \omega_{N-1}}{s(1-s)}.
	\end{equation*}
	
%	\textcolor{red}{Hence we finally obtain 
%	\begin{equation}\label{fracMeanCurvatureUpperBound}
%		H_{\capM_{\varepsilon},s}(q) \leq  c(N,s) \|\nabla'^2 w_{\varepsilon}\|_{C^0}^s =  c(N,s) \|\nabla'^2w\|_{C^0}^s \,\varepsilon^s \leq c'(N,s) \varepsilon^s
%	\end{equation}
%	where $c'(N,s) \coloneqq 2^sc(N,s) (N-1)^{s/2}$.}
	
	Next we compute the fractional mean curvature $H_{\capM^{\varepsilon,t},s}(q)$ at $q$ by using Estimate \eqref{fracMeanCurvatureMBound} of the fractional mean curvature $H_{\widetilde{\capM}^{\varepsilon,t},s}(q)$ at $q$. Indeed, from the construction of $\capM^{\varepsilon,t}$ and $\widetilde{\capM}^{\varepsilon,t}$, we have that, by choosing a proper orientation, $\capA^{\widetilde{\capM}^{\varepsilon,t}}_e(q) \subset \capA^{\capM^{\varepsilon,t}}_e(q)$	and thus we obtain
	\begin{align}
		- H_{\capM^{\varepsilon,t},s}(q) &= - H_{\widetilde{\capM}^{\varepsilon,t},s}(q) \nonumber\\
		&\qquad + \int_{\mathR^N}\frac{\chi_{\capA^{\capM^{\varepsilon,t}}_e(q)}(x) - \chi_{\capA^{\widetilde{\capM}^{\varepsilon,t}}_e(q)}(x) + \chi_{\capA^{\widetilde{\capM}^{\varepsilon,t}}_i(q)}(x) - \chi_{\capA^{\capM^{\varepsilon,t}}_i(q)}(x)}{|x-q|^{N+s}}\,dx \nonumber\\
		&= - H_{\widetilde{\capM}^{\varepsilon,t},s}(q) + \int_{\mathR^N} \frac{\chi_{\capA^{\capM^{\varepsilon,t}}_e(q)\setminus \capA^{\widetilde{\capM}^{\varepsilon,t}}_e(q)}(x) + \chi_{\capA^{\widetilde{\capM}^{\varepsilon,t}}_i(q) \setminus \capA^{\capM^{\varepsilon,t}}_i(q)}(x)}{|x-q|^{N+s}} \,dx \nonumber\\
		&= - H_{\widetilde{\capM}^{\varepsilon,t},s}(q) +  2\int_{ \capA^{\capM^{\varepsilon,t}}_e(q) \setminus \capA^{\widetilde{\capM}^{\varepsilon,t}}_e(q)} \frac{1}{|x-q|^{N+s}} \,dx. \label{fracMeanCurvatureTildeH}
	\end{align}
%	\textcolor{red}{Indeed, from the construction of $\capM_{\varepsilon}$ and $\widetilde{\capM}_{\varepsilon}$, we first have
%	\begin{align*}
%		H_{\widetilde{\capM}_{\varepsilon},s}(q) &= H_{\capM_{\varepsilon},s}(q) + \int_{\mathR^N}\frac{-\chi_{\capA^{\widetilde{\capM}_{\varepsilon}}_e(q)}(x) + \chi_{\capA^{\capM_{\varepsilon}}_e(q)}(x) - \chi_{\capA^{\capM_{\varepsilon}}_i(q)}(x) + \chi_{\capA^{\widetilde{\capM}_{\varepsilon}}_i(q)}(x)}{|x-q|^{N+s}}\,dx \nonumber\\
%		&= H_{\capM_{\varepsilon},s}(q) - \int_{\mathR^N} \frac{\chi_{\capA^{\widetilde{\capM}_{\varepsilon}}_e\setminus \capA^{\capM_{\varepsilon}}_e(q)}(x) + \chi_{\capA^{\capM_{\varepsilon}}_i \setminus \capA^{\widetilde{\capM}_{\varepsilon}}_i(q)}(x)}{|x-q|^{N+s}} \,dx \nonumber\\
%		&= H_{\capM_{\varepsilon},s}(q) -  2\int_{ \capA^{\widetilde{\capM}_{\varepsilon}}_e \setminus \capA^{\capM_{\varepsilon}}_e(q)} \frac{1}{|x-q|^{N+s}} \,dx.
%	\end{align*}
%	From the definitions of $\widetilde{\capM}_{\varepsilon}$ and $\capM_{\varepsilon}$, we have that $\capA^{\widetilde{\capM}_{\varepsilon}}_e \setminus \capA^{\capM_{\varepsilon}}_e(q)$ contains an open ball of radius $\varepsilon^{\beta}/2$ and the distance between $q$ and that ball is at most $d$. Thus, from \eqref{fracMeanCurvatureUpperBound}, we obtain
%	\begin{equation}
%		H_{\widetilde{\capM}_{\varepsilon},s}(q) \leq  c'(N,s) \varepsilon^s - \frac{2\omega_{N}}{d^{N+s}}\varepsilon^{N\beta} < c'(N,s) \varepsilon^s - c_0(N,s) \varepsilon^{N\beta - }.
%	\end{equation}
%	}
	Recalling that $ \capA^{\capM^{\varepsilon,t}}_e(q) \setminus \capA^{\widetilde{\capM}^{\varepsilon,t}}_e(q)$ contains the subgraph $\{0\leq x_N \leq v_{\varepsilon}(x'-b')\}$ and the subgraph contains the cylinder of height $\phi(\varepsilon)^{\beta}$ with the base of radius $1/16$, we have
	\begin{equation*}
		\left| \capA^{\capM^{\varepsilon,t}}_e(q) \setminus \capA^{\widetilde{\capM}^{\varepsilon,t}}_e(q) \right| \geq c'\,\phi(\varepsilon)^{\beta}.
	\end{equation*}
	where a constant $c'=c'(N)>0$ depends only on $N$. Moreover, we observe that the distance between $q$ and the cylinder is less than, at most, $2+\phi(\varepsilon)^{\beta}$ and this is bounded from above by some constant depending only on $N$, $s$, and $\beta$. Hence from \eqref{fracMeanCurvatureMBound} and \eqref{fracMeanCurvatureTildeH} and by recalling the choice of $\phi$, we obtain
	\begin{align}
		- H_{\capM^{\varepsilon,t},s}(q) &\geq -c \, \phi(\varepsilon)^s + \frac{2c'}{(2+\phi(\varepsilon)^{\beta})^{N+s}} \,\phi(\varepsilon)^{\beta} \nonumber\\
		&\geq -c \, \phi(\varepsilon)^s + c'' \phi(\varepsilon)^{\beta} \nonumber\\
		&= \phi(\varepsilon)^{\beta}\left( - c \, \phi(\varepsilon)^{s-\beta} + c''\right) \label{fracMeanCurvatureTildeH02}
	\end{align}
	where $c''>0$ is a constant depending only on $N$, $s$, and $\beta$. Since $0 < \beta < s$ and $\phi(\varepsilon) \downarrow 0$ as $\varepsilon \downarrow 0$, we choose $\varepsilon_1=\varepsilon_1(N,s,\beta) \in I_{\phi} \cap (0,\, \frac{1}{100})$ so small that the right-hand side of \eqref{fracMeanCurvatureTildeH02} is positive for any $\varepsilon \in (0,\,\varepsilon_1]$. Therefore, from \eqref{fracMeanCurvatureTildeH02}, we obtain that $H_{\capM^{\varepsilon,t}, s}(q) < 0$ for $\varepsilon \in (0,\,\varepsilon_1]$. 
	
	Now we set $\varepsilon_2 \coloneqq \min\{\varepsilon_1,\,\varepsilon_d\}$. Since $r(\varepsilon_d) \geq d$ and $\delta(\varepsilon_2) < \frac{1}{2}$, we may observe that $d(\varepsilon_2) \geq d$ and $\capM_3^{\varepsilon_2,t} \cap \Gamma_1 = \emptyset$ for any $t \in (0,\,\varepsilon_2]$.  For our convenience, we denote $\varepsilon_2$ by $\varepsilon$ in the sequel. 
	
	We then slide the barrier $\capM^{\varepsilon,t}$ from above, i.e., we vary the parameter $t$ stating at $\varepsilon$ until $\capM^{\varepsilon,t}$ touches the critical point $\capM$. To prove the claim, we assume by contradiction that there exists $t_1 \in (0,\,\varepsilon]$ such that $\capM \cap \capM_3^{\varepsilon,t_1} \neq \emptyset$ and $\capM \cap \capM_3^{\varepsilon,t} = \emptyset$ for any $t \in (t_1,\,\varepsilon]$. We pick up a point $q_{\varepsilon,t_1} \in \capM \cap \capM_3^{\varepsilon, t_1}$. Notice that
	\begin{equation*}
		\{(x',\,x_N) \mid -d < x_N < 0\} \cap \capM_2^{\varepsilon, t_1} = \emptyset
	\end{equation*}
	since $d(\varepsilon) \geq d$. See Figure \ref{figureComparisonCriticalBarrier} to favor the intuition in dimension 2.
	
	\begin{figure}[h]
		\begin{center}
			\includegraphics[keepaspectratio, scale=0.50]{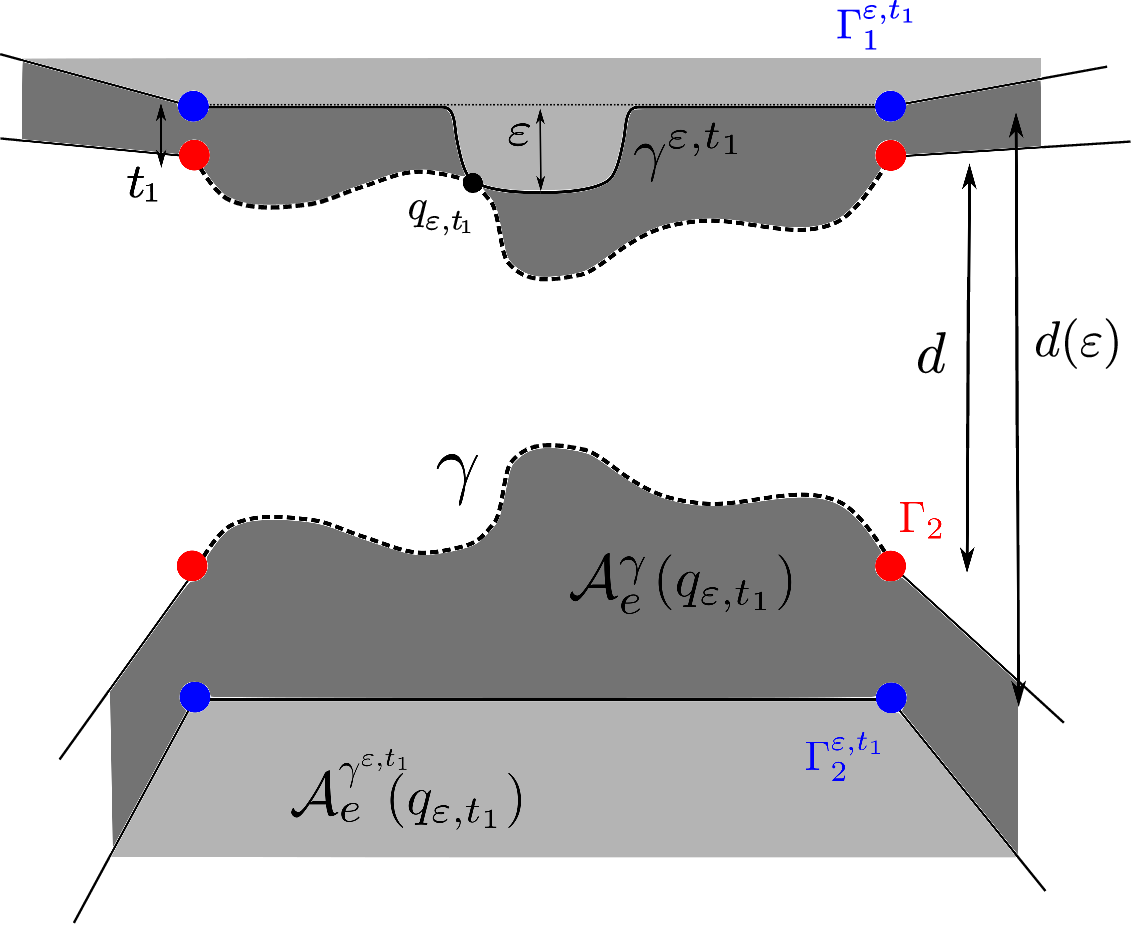}
		\end{center}
		\caption{The critical point $\gamma$ depicted with dashed lines and the barrier $\gamma^{\varepsilon, t_1}$ with black line. $\gamma$ touches $\gamma^{\varepsilon, t_1}$ at $q_{\varepsilon,t_1}$ from above. The exterior $\capA^{\gamma^{\varepsilon, t_1}}_e(q_{\varepsilon,t_1})$ of $\gamma^{\varepsilon, t_1}$ is depicted in light gray and the exterior $\capA_e^{\gamma}(q_{\varepsilon,t_1})$ of $\gamma$ in both light and dark gray.}
		\label{figureComparisonCriticalBarrier}
	\end{figure}
	
	Since $\capM$ is a critical point of $\Pers$ under normal variations, we obtain
	\begin{equation*}
		H_{\capM,s}(q_{\varepsilon,t_1}) = 0.
	\end{equation*}
	From Theorem \ref{theoremTwoHyperplanesNotMinimal} and the above argument, we obtain that the touching point $q_{\varepsilon,t_1} \coloneqq (q'_{\varepsilon,t_1},\,q_{\varepsilon,t_1}^N) \in \mathR^{N-1} \times \mathR$ satisfies $|q'_{\varepsilon,t_1}| < \delta(\varepsilon)$ and thus $H_{\capM^{\varepsilon, t_1},s}(q_{\varepsilon,t_1}) < 0$. Moreover, from the construction of $\capM_1^{\varepsilon, t_1}$, we have, by choosing a proper orientation, that
	\begin{equation*}
		|\capA^{\capM}_e(q_{\varepsilon,t_1}) \setminus \capA^{\capM^{\varepsilon, t_1}}_e(q_{\varepsilon,t_1})| > 0 \quad \text{and} \quad |\capA^{\capM^{\varepsilon, t_1}}_e(q_{\varepsilon,t_1}) \setminus \capA^{\capM}_e(q_{\varepsilon,t_1})| = 0.
	\end{equation*}
	Therefore, we obtain
	\begin{align}
		0 &= H_{\capM,s}(q_{\varepsilon,t_1}) - H_{\capM^{\varepsilon, t_1},s}(q_{\varepsilon,t_1}) + H_{\capM^{\varepsilon, t_1},s}(q_{\varepsilon,t_1}) \nonumber\\
		&< \int_{\mathR^N} \frac{\chi_{\capA^{\capM}_i(q_{\varepsilon,t_1})}(x) - \chi_{\capA^{\capM^{\varepsilon, t_1}}_i(q_{\varepsilon,t_1})}(x) + \chi_{\capA^{\capM^{\varepsilon, t_1}}_e(q_{\varepsilon,t_1})}(x) - \chi_{\capA^{\capM}_e(q_{\varepsilon,t_1})}(x)}{|x-q_{\varepsilon,t_1}|^{N+s}}\,dx + 0 \nonumber\\
		&= -2\int_{\capA^{\capM}_e(q_{\varepsilon,t_1}) \setminus \capA^{\capM^{\varepsilon, t_1}}_e(q_{\varepsilon,t_1})} \frac{1}{|x-q_{\varepsilon,t_1}|^{N+s}}\,dx < 0,
	\end{align}
	which is a contradiction. We thus conclude that we can slide the barrier $\capM^{\varepsilon,t}$ until the boundary $\Gamma_1^{\varepsilon,t} = \partial \capM_3^{\varepsilon,t}$ coincides with the boundary $\Gamma_1 = \partial \capM_1$. By symmetry, we can slide the barrier from below and do the same argument. 
	
	Therefore we obtain that two open half-balls of radius $\varepsilon_2$ are contained in a set enclosed by $\capM$ and the union of $\capC \cap \{x_N=0\}$ and $\capC \cap \{x_N=-d\}$. 
	
\end{proof}

As a consequence of Lemma \ref{theoremPoppingCriticalPtwithTwoBdry}, we now prove Theorem \ref{theoremConnectednessCriticalPtsN3}.

\begin{proof}[Proof of Theorem \ref{theoremConnectednessCriticalPtsN3}]
%	The idea of the proof of Theorem \ref{theoremConnectednessCriticalPtsN3} is to slide the barrier that we construct in Lemma \ref{theoremPoppingCriticalPtwithTwoBdry} until it touches the critical point and to compare the Euler-Lagrange equation of the barrier to that of the critical point.
%we slide the barrier $\capM^{\varepsilon_2,t}$ from above, i.e., we vary the parameter $t$ starting at $\varepsilon_2$ until $\capM^{\varepsilon_2,t}$ touches the critical point $\capM$. Then, by employing the same argument as in the proof of Lemma \ref{theoremPoppingCriticalPtwithTwoBdry}, 
	Assume that $\varepsilon_2$ and $\widetilde{\capM}^{\varepsilon,t}$ are given in the proof of Lemma \ref{theoremPoppingCriticalPtwithTwoBdry} for $\varepsilon \in (0,\,\varepsilon_2]$ and $t \in (0,\,\varepsilon]$. From the definition of $\varepsilon_2$, we can choose $d'>0$ so small that $d' < (2(N-1))^{-1}$ and that $\varepsilon_2$ can be chosen independently of $d$ for any $d \in (0,\, d')$. Moreover, if necessary, we may assume that $\varepsilon_2\,\phi(\varepsilon_2) < (2(N-1))^{-1}$, which is still independent of $d$.
	
	Let $\capM$ be the critical point chosen in Theorem \ref{theoremConnectednessCriticalPtsN3}. We set $d_0 \coloneqq \min\{d',\, \varepsilon_2\}$. From the choice of $\phi$ and $\varepsilon_2$, we have that $d_0\, \phi(d_0) < (2(N-1))^{-1}$. Then we observe that $d(\varepsilon_2) - t = ((N-1)\phi(\varepsilon_2))^{-1} - t > d$ for any $t \in (0,\,\varepsilon_2]$ and thus we have that
	\begin{equation*}
		\Gamma_2^{\varepsilon_2,t} \cap \{(x',\,x_N) \mid -d< x_N <0\} = \emptyset
	\end{equation*}
	for any $t \in (0,\,\varepsilon_2]$ and any $d < d_0$.
	
	Now, by Lemma \ref{theoremPoppingCriticalPtwithTwoBdry}, we find that we can slide the barrier $\capM^{\varepsilon_2,t}$ until the parameter $t$ reaches $0$. Thus, by combining this with Theorem \ref{theoremTwoHyperplanesNotMinimal}, we obtain that
	\begin{align*}
		\capM &\subset \left( \overline{\capC} \setminus \{(x',\,x_N) \mid |x'| < \varepsilon_2 \} \right) \cap \{(x',\,x_N) \mid -d \leq x_N \leq 0\} \nonumber\\
		&= \{(x',\,x_N) \mid \varepsilon_2 \leq |x'| \leq 1, \, -d \leq x_N \leq 0 \}
	\end{align*}
	for any $d < d_0$. 
	
	If $N=2$, then, since $\Gamma_i$ consists of two distinct points for $i \in \{1,\,2\}$, by a simple geometric argument, we conclude that the critical point $\capM$ is disconnected for any $d \in (0,\,d_0)$. Moreover, from the construction of the barrier, we obtain that there exist two connected components $\capM_1$ and $\capM_2$ of $\capM$ such that $\dist(\capM_1,\capM_2) \geq \varepsilon_2$ and $\capM_i$ intersects both $\Gamma_1$ and $\Gamma_2$ for each $i \in \{1,\, 2\}$ at its boundary (see also Remark \ref{remarkCriticalPtsDsmall}). 
	
	If $N \geq 3$, then, by using a homology theory, we conclude that $\Gamma_1$ and $\Gamma_2$ are in the same connected component of the critical point $\capM$ for any $d \in (0,\,d_0)$ (see \cite{stackexchange}). Indeed, we assume by contradiction that there exists a connected component $\capM_0$ of $\capM$ with $\partial \capM_0 = \Gamma_1$. Taking the fundamental class $[\Gamma_1] \in H_{N-2}(\Gamma_1)$ where $H_k(\capS)$ is the $k$-th homology group of $\capS$, we may have that the image in $H_{N-2}(\capM_0)$ of $[\Gamma_1]$ by the induced map of homology from the inclusion $i: \Gamma_1 \to \capM_0$ does not vanish because $\capM_0 \subset \capM \subset A_{\varepsilon_2}$ and $A_{\varepsilon_2}$ deformation-retracts to $\Gamma_1 \simeq \mathS^{N-2}$ where
	\begin{equation*}
		A_{\varepsilon_2} \coloneqq \{(x',\, x_N) \mid \varepsilon_2 \leq |x'| \leq 1, -d \leq x_N \leq 0 \}.
	\end{equation*}
	 However, since $[\Gamma_1]$ is the boundary of $[\capM_0]$ and by using an exact homology sequence of the pair $(\capM_0, \,\Gamma_1)$, we obtain the contradiction. 
\end{proof}
%, from Proposition \ref{propositionConnectednessManiWithTwoDisjointBdry} in Appendix \ref{sectionTopManiWithDisjointBoundaries},
\begin{remark}\label{remarkCriticalPtsDsmall}
	Combining Remark \ref{remarkCriticalPts2D} with Lemma \ref{theoremPoppingCriticalPtwithTwoBdry} and Theorem \ref{theoremConnectednessCriticalPtsN3}, we may observe that two possible critical points of $\Pers$ in dimension 2 whose boundary is $\Gamma_1 \cup \Gamma_2 = \{(\pm 1,\,d), \, (\pm 1, \,-d)\}$ are depicted in Figure \ref{figure7}.
	
	\begin{figure}[!htb]
		\begin{tabular}{cc}

			\begin{minipage}{0.48\hsize}
				\centering
				\includegraphics[keepaspectratio, scale=0.35]{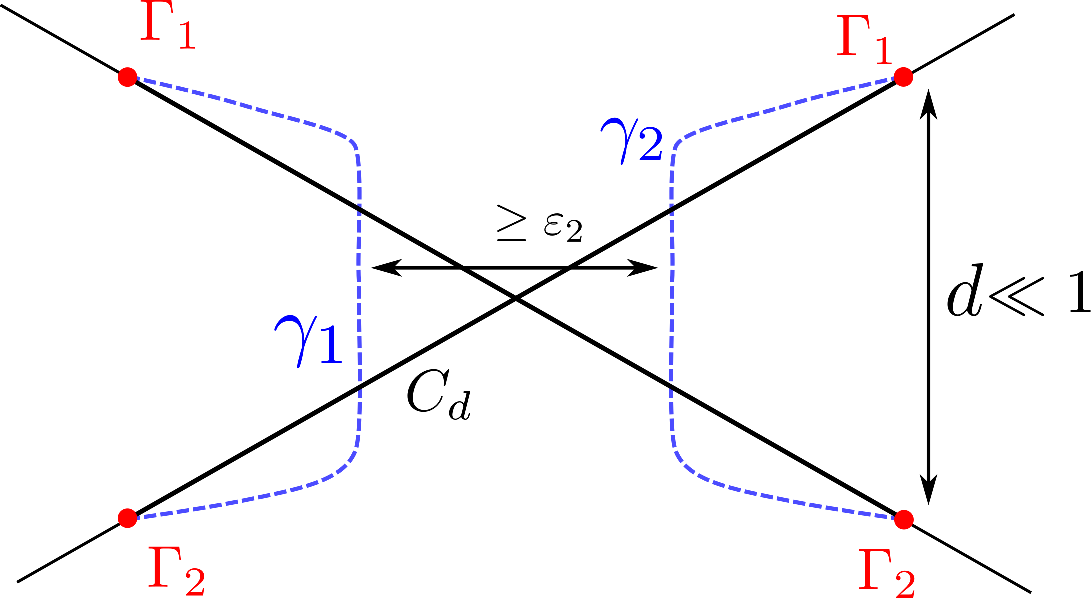}
			\end{minipage}
			
			\begin{minipage}{0.48\hsize}
				\centering
				\includegraphics[keepaspectratio, scale=0.35]{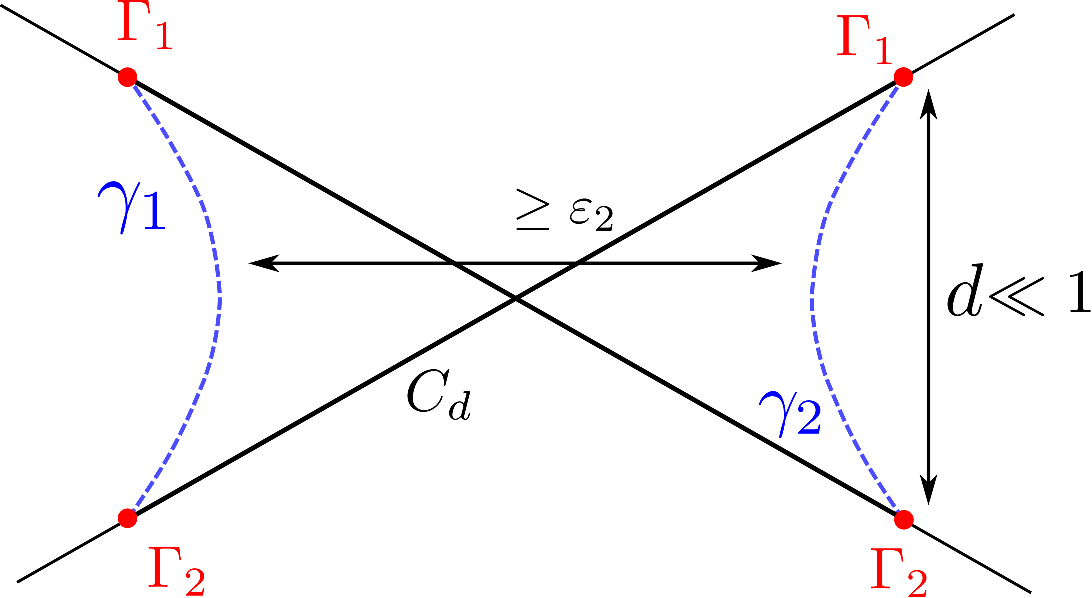}
			\end{minipage}
			
		\end{tabular}
		\caption{Two possible critical points $\gamma$ of $\Pers$ in dimension 2 with $\partial \gamma = \Gamma_1 \cup \Gamma_2$ are shown with dashed lines and the cone $C_d$ defined in \eqref{cone2DPositiveorNegativeFracMC} is shown with crossed lines. On the right, $\gamma$ does not intersect with $C_d$ except at their boundaries $\Gamma_1$ and $\Gamma_2$. In both figures, two distinct connected components $\gamma_1$ and $\gamma_2$ of $\gamma$ are placed at mutually positive distance of at least $\varepsilon_2>0$.}
		\label{figure7}
	\end{figure} 
	
\end{remark}

Finally in this section, we prove Theorem \ref{theoremDisconnectednessCriticalPts}. The idea of the proof is basically the same as the one in \cite[Theorem 1.2]{DOV22}, i.e., we use the ``sliding method'' that is developed by Dipierro, Savin, and Valdinoci in \cites{DSV16, DSV17, DSV20}. 
\begin{proof}[Proof of Theorem \ref{theoremDisconnectednessCriticalPts}]
	Let $\capM$ be the critical point selected in Theorem \ref{theoremDisconnectednessCriticalPts} and we set $\Gamma \coloneqq \Gamma_1 \cup \Gamma_2$.
	
	Given $t \in \mathR$ and $\alpha \in (0,\,1)$, we consider the open ball $B_{d^{\alpha}/2}(p_{t,d})$ where $p_{t,d} \coloneqq (te'_1,\,\frac{-d}{2}) \in \mathR^{N-1} \times \mathR$ and $e'_1 \coloneqq (1,\,0,\cdots,0) \in \mathR^{N-1}$. Here we take $d$ conveniently large so that $d - d^{\alpha} > 100$. Then we slide the ball from left to right until it touches the critical point $\capM$, which means that we vary $t$ from $t = -\infty$ to $t = +\infty$. Note that $B_{d^{\alpha}/2}(p_{t,d}) \subset \capC^c$ for $|t|>1+d^{\alpha}/2$ and $B_{d^{\alpha}/2}(p_{t,d}) \cap \Gamma = \emptyset$ for any $t$. To prove the claim, we suppose by contradiction that there exists $t_0 \in \mathR$ such that $\overline{B}_{d^{\alpha}/2}(p_{t,d}) \cap \capM = \emptyset$ for $t < t_0$ and $\partial B_{d^{\alpha}/2}(p_{t_0,d}) \cap \capM \neq \emptyset$.
	
	We choose a point $q \in \partial B_{d^{\alpha}/2}(p_{t_0,d}) \cap \capM$. Note that, due to Theorem \ref{theoremTwoHyperplanesNotMinimal}, $q \in \capC$. By the Euler-Lagrange equation, we have that
	\begin{equation}\label{eulerLagrageThm1.3}
		H_{\capM,s}(q) = 0.
	\end{equation}
	Moreover, by choosing a proper orientation, we can choose the interior $\capA_i^{\capM}(q)$ and exterior $\capA_e^{\capM}(q)$ at $q$ of $\capM$ in such a way that 
	\begin{equation}\label{inclusionSlidingBallInterior}
		B_{\frac{d^{\alpha}}{2}}(p_{t_0,d}) \subset \capA_i^{\capM}(q) \quad \text{and} \quad \capA_e^{\capM}(q) = \mathR^N \setminus \left( \capA_i^{\capM}(q) \cup \capM \right).
	\end{equation}
	We now consider the symmetric ball of $B_{d^{\alpha}/2}(p_{t_0,d})$ with respect to $q$ and we denote it by $\widetilde{B} \coloneqq B_{d^{\alpha}/2}(\widetilde{p}_{t_0,d})$ where $\widetilde{p}_{t_0,d} \coloneqq p_{t_0,d} + 2(q-p_{t_0,d})$.
	
	We define a cylinder $S_d$ as
	\begin{equation*}
		S_d \coloneqq \{(x',\,x_N) \mid |x'| < 2,\, -d < x_N < 0\}.
	\end{equation*}
	Notice that $\capC \cap \{(x',\,x_N) \mid -d< x_N < 0\} \subset S_d$ and $\capM \subset S_d$ thanks to Theorem \ref{theoremTwoHyperplanesNotMinimal}. From the symmetry of the balls, we have
	\begin{equation*}
		\int_{S_d \cap B_{\frac{d^{\alpha}}{2}}(p_{t_0,d})} \frac{dx}{|x-q|^{N+s}} = \int_{S_d \cap \widetilde{B}} \frac{dx}{|x-q|^{N+s}}
	\end{equation*}
	and therefore, from \eqref{inclusionSlidingBallInterior},
	\begin{align}
		\int_{S_d} \frac{\chi_{\capA_i^{\capM}(q)}(x) - \chi_{\capA_e^{\capM}(q)}(x)}{|x-q|^{N+s}}\,dx &= \int_{S_d \cap B_{\frac{d^{\alpha}}{2}}(p_{t_0,d})} \frac{\chi_{\capA_i^{\capM}(q)}(x) - \chi_{\capA_e^{\capM}(q)}(x) }{|x-q|^{N+s}}\,dx \nonumber\\
		&\qquad + \int_{S_d \cap \widetilde{B}} \frac{\chi_{\capA_i^{\capM}(q)}(x) - \chi_{\capA_e^{\capM}(q)}(x) }{|x-q|^{N+s}}\,dx \nonumber\\
		&\qquad \quad + \int_{S_d \setminus \left( B_{\frac{d^{\alpha}}{2}}(p_{t_0,d}) \cup \widetilde{B}\right)} \frac{\chi_{\capA_i^{\capM}(q)}(x) - \chi_{\capA_e^{\capM}(q)}(x) }{|x-q|^{N+s}}\,dx \nonumber\\
		&\geq \int_{S_d \cap B_{\frac{d^{\alpha}}{2}}(p_{t_0,d})} \frac{dx}{|x-q|^{N+s}} - \int_{S_d \cap \widetilde{B}} \frac{dx}{|x-q|^{N+s}} \nonumber\\
		&\qquad - \int_{S_d \setminus \left( B_{\frac{d^{\alpha}}{2}}(p_{t_0,d}) \cup \widetilde{B}\right)} \frac{dx}{|x-q|^{N+s}} \nonumber\\
		&\geq - \int_{S_d \setminus \left( B_{\frac{d^{\alpha}}{2}}(p_{t_0,d}) \cup \widetilde{B}\right)} \frac{dx}{|x-q|^{N+s}}. \label{computationFracPeriThm1.3}
	\end{align}
	By employing the result in \cite[Lemma 3.1]{DSV16} with $R = d^{\alpha}/2$ and $\lambda = d^{-\frac{\alpha}{2}}$, we obtain
	\begin{equation*}
		\int_{B_{d^{\frac{\alpha}{2}}}(q) \setminus \left( B_{\frac{d^{\alpha}}{2}}(p_{t_0, d}) \cup \widetilde{B}\right) } \frac{dx}{|x-q|^{N+s}} \leq C_0 \, d^{-\frac{1+s}{2}\alpha}
	\end{equation*}
	where $C_0>0$ is a constant depending only on $N$ and $s$. As a consequence, we obtain
	\begin{align}
		\int_{S_d \setminus \left( B_{\frac{d^{\alpha}}{2}}(p_{t_0,d}) \cup \widetilde{B}\right)} \frac{dx}{|x-q|^{N+s}} &\leq \int_{B_{d^{\frac{\alpha}{2}}}(q) \setminus \left( B_{\frac{d^{\alpha}}{2}}(p_{t_0,d}) \cup \widetilde{B}\right)} \frac{dx}{|x-q|^{N+s}} \nonumber\\
		&\qquad + \int_{S_d \setminus B_{d^{\frac{\alpha}{2}}}(q) } \frac{dx}{|x-q|^{N+s}} \nonumber\\
		&\leq C_0 \, d^{-\frac{1+s}{2}\alpha} + \int_{\mathR^N \setminus B_{d^{\frac{\alpha}{2}}}(q) } \frac{dx}{|x-q|^{N+s}} \nonumber\\
		&\leq C_0 \, d^{-\frac{1+s}{2}\alpha} + C_1 \, d^{-\frac{s}{2}\alpha} \leq C_2 \, d^{-\frac{s}{2}\alpha} \label{computationFracPeri02Thm1.3}
	\end{align} 
	where $C_2 \coloneqq C_0 + C_1$ is a constant depending only on $N$ and $s$. From \eqref{computationFracPeriThm1.3} and \eqref{computationFracPeri02Thm1.3}, we have
	\begin{align}
		\int_{\mathR^N} \frac{\chi_{\capA_i^{\capM}(q)}(x) - \chi_{\capA_e^{\capM}(q)}(x)}{|x-q|^{N+s}}\,dx &= \int_{S_d} \frac{\chi_{\capA_i^{\capM}(q)}(x) - \chi_{\capA_e^{\capM}(q)}(x)}{|x-q|^{N+s}}\,dx \nonumber\\
		&\qquad + \int_{S_d^c} \frac{\chi_{\capA_i^{\capM}(q)}(x) - \chi_{\capA_e^{\capM}(q)}(x)}{|x-q|^{N+s}}\,dx \nonumber\\
		&\geq -C_2 \, d^{-\frac{s}{2}\alpha} + \int_{S_d^c} \frac{\chi_{\capA_i^{\capM}(q)}(x) - \chi_{\capA_e^{\capM}(q)}(x)}{|x-q|^{N+s}}\,dx. \label{computationFracPeri03Thm1.3}
	\end{align}
	
	Now we consider the contributions from $\capA_i^{\capM}(q)$ and $\capA_e^{\capM}(q)$ in $S_d^c$. We now define $C_{\Gamma}(q)$ by a (filled) cone of vertex $q$ whose boundary passes through $\Gamma$. Moreover we define $C_{S_d}(q)$ by a (filled) cone of vertex $q$ whose boundary passes through	\begin{equation*}
		\partial S_d \cap \{(x',\,x_N) \mid x_N = 0\} \quad \text{and} \quad \partial S_d \cap \{(x',\,x_N) \mid x_N = -d\}.
	\end{equation*}
	From the definitions of $S_d$ and $\Gamma$, we observe that
	\begin{equation*}
		C_{\Gamma}(q) \subset C_{S_d}(q).
	\end{equation*} 
	We now set $\widehat{C}_{\Gamma}(q) \coloneqq C_{\Gamma}(q) \cap \{(x',\,x_N) \mid x_N > 0 \; \text{or} \; x_N < -d\}$. We then rotate $\widehat{C}_{\Gamma}(q)$ by angle $\pi/2$ or $-\pi/2$ with respect to the straight line parallel to the $x_1$-axis passing trough $q$ (if $N=2$, then we just rotate $\widehat{C}_{\Gamma}(q)$ by angle $\pi/2$ or $-\pi/2$ with respect to $q$). Since we choose $d$ so large that $d - d^{\alpha} > 100$, we obtain that
	\begin{equation*}
		R(\widehat{C}_{\Gamma}(q)) \subset S_d^c \cap \capA_i^{\capM}(q) \cap C_{S_d}(q)^c
	\end{equation*}
	where $R(\widehat{C}_{\Gamma}(q))$ is an image of $\widehat{C}_{\Gamma}(q)$ by the rotation map $R:\mathR^N \to \mathR^N$ in the above. See Figure \ref{figureSlidingBallThm1.3} for an intuitive understanding. Then, observing that $R(q) = q$ and
	\begin{equation*}
		S_d^c \cap \capA_e^{\capM}(q) = \widehat{C}_{\Gamma}(q)
	\end{equation*}
	and by a change of variables, we have
	\begin{equation*}
		\int_{R(\widehat{C}_{\Gamma}(q))} \frac{dx}{|x-q|^{N+s}} = \int_{\widehat{C}_{\Gamma}(q)} \frac{dx}{|x-q|^{N+s}} = \int_{S_d^c \cap \capA_e^{\capM}(q)} \frac{dx}{|x-q|^{N+s}}.
	\end{equation*}
	
	\begin{figure}[h]
		\begin{center}
			\includegraphics[keepaspectratio, scale=0.55]{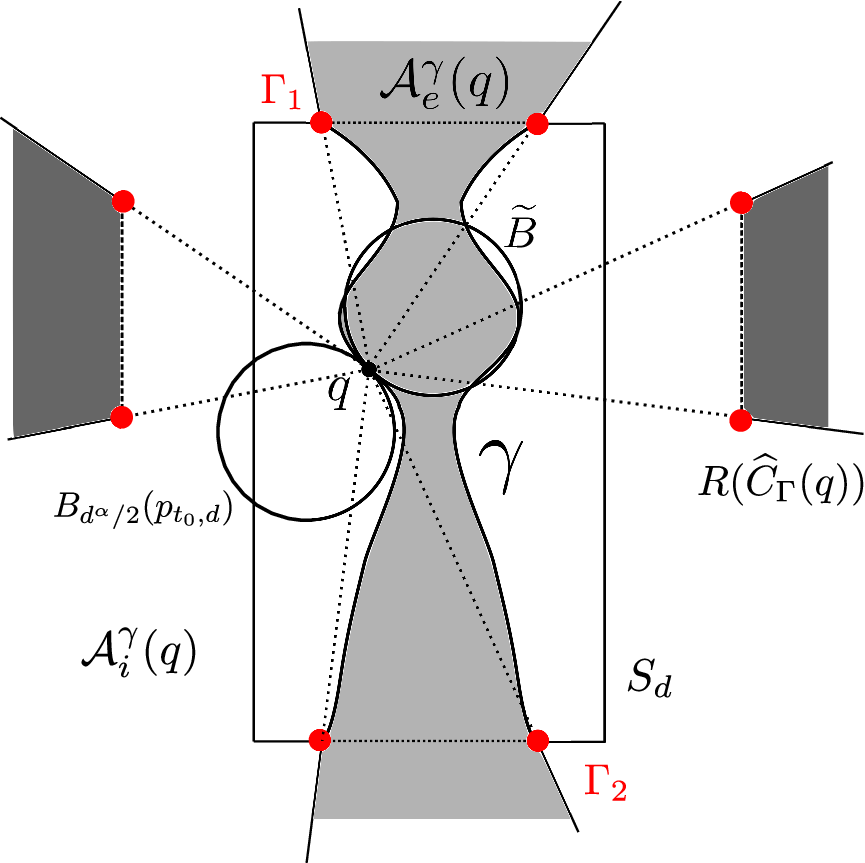}
		\end{center}
		\caption{The touching between the ball $B_{d^{\alpha}/2}(p_{t_0,d})$ and its symmetric ball $\widetilde{B}$ at $q$. The image of the set $\widehat{C}_{\Gamma}(q)$ by the rotation map $R$ is depicted in dark gray and the set $\capA_e^{\gamma}(q)$ in light gray.}
		\label{figureSlidingBallThm1.3}
	\end{figure}
	
	From the definitions of $S_d$ and the rotation map $R$, we
	can choose an open ball outside $S_d$ and $R(\widehat{C}_{\Gamma}(q))$ but close to $q$, i.e., we have
	\begin{equation*}\label{smallBallContainedSd}
		B_1(q+5e_1) \subset \left( S_d^c \cap \capA_i^{\capM}(q)  \right) \setminus R(\widehat{C}_{\Gamma}(q)) 
	\end{equation*}
	where we set $e_1 \coloneqq (1,\,0,\cdots,0) \in \mathR^N$. Thus, we obtain
	\begin{align*}
		\int_{S_d^c} \frac{\chi_{\capA_i^{\capM}(q)}(x) - \chi_{\capA_e^{\capM}(q)}(x)}{|x-q|^{N+s}}\,dx &= \int_{S_d^c \cap \capA_i^{\capM}(q)} \frac{dx}{|x-q|^{N+s}} - \int_{\widehat{C}_{\Gamma}(q)} \frac{dx}{|x-q|^{N+s}} \nonumber\\
		&\geq \int_{S_d^c \cap \capA_i^{\capM}(q) \cap R(\widehat{C}_{\Gamma}(q))} \frac{dx}{|x-q|^{N+s}} \nonumber\\
		&\qquad + \int_{\left( S_d^c \cap \capA_i^{\capM}(q) \right) \setminus R(\widehat{C}_{\Gamma}(q))} \frac{dx}{|x-q|^{N+s}} \nonumber\\
		&\qquad \quad - \int_{\widehat{C}_{\Gamma}(q)} \frac{dx}{|x-q|^{N+s}} \nonumber\\
		&\geq \int_{B_1(q+5e_1)} \frac{dx}{|x-q|^{N+s}} \nonumber\\
		&= \int_{B_1(5e_1)} \frac{dx}{|x|^{N+s}} \eqqcolon C_3 > 0
	\end{align*}
	where $C_3$ depends only on $N$ and $s$. This with \eqref{computationFracPeri03Thm1.3} leads to
	\begin{equation*}
		H_{\capM,s}(q) = \int_{\mathR^N} \frac{\chi_{\capA_i^{\capM}(q)}(x) - \chi_{\capA_e^{\capM}(q)}(x)}{|x-q|^{N+s}}\,dx \geq -C_2 \,d^{-\frac{s}{2}\alpha} + C_3.
	\end{equation*}
	Therefore, there exists $d_1=d_1(N,s)>0$ such that $H_{\capM,s}(q) > 0$ for any $d > d_1$ and this contradicts the Euler-Lagrange equation \eqref{eulerLagrageThm1.3}.
\end{proof}

\begin{remark}
	From Lemma \ref{theoremPoppingCriticalPtwithTwoBdry} and the choice of $\varepsilon_2$ in the proof of Lemma \ref{theoremPoppingCriticalPtwithTwoBdry}, we also obtain that, for sufficiently large $d$, a set enclosed by $\capM$ and the union of $\capC \cap \{x_N=0\}$ and $\capC \cap \{x_N=-d\}$ contains two half-balls of radius $\varepsilon_2 \thickapprox \phi^{-1}(d^{-1})$ where $\phi^{-1}$ is as in the proof of Lemma \ref{theoremPoppingCriticalPtwithTwoBdry}.
\end{remark}

\begin{center}
	\textbf{{\small Acknowledgments}}
\end{center}
The author would like to thank Brian Seguin for fruitful discussions and many comments. The author was supported by the DFG Collaborative Research Center TRR 109, ``Discretization in Geometry and Dynamics''.

%%%%%%%%%%%%%%%%%%%%%%%%%%%%%%%BIBLIOGRAPHY%%%%%%%%%%%%%%%%%%%%%%%%%%%%%%%%%%%
\bibliographystyle{plain}
\bibliography{biblio_fractional_mini_with_bdry_onoue}

\end{document}